\documentclass[leqno,a4paper]{amsart}

\usepackage{amsmath,amsfonts,amsthm,amssymb,dsfont}
\usepackage{relsize} 

\usepackage{	tgheros}
\usepackage{microtype}
\usepackage{color,soul}
\usepackage[]{mdframed} 
\usepackage{xcolor} 
\usepackage{color}
\usepackage{soul} 
\usepackage{cite}

\usepackage{mathrsfs}
\usepackage{enumerate, xspace}
\usepackage{graphicx}
\usepackage{color}
\usepackage[makeroom]{cancel} 

    \usepackage[all, knot]{xy}
    \xyoption{arc} 

\usepackage{pgf,tikz}\usepackage{mathrsfs}\usetikzlibrary{arrows}
\usepackage{enumitem} 
\usepackage{tikz,lipsum,lmodern} 
\usepackage{tcolorbox} 
\usepackage{verbatim}

\usepackage{caption}
\usepackage{subcaption}
\usepackage[pdfauthor={Arman Darbinyan  and},pdftitle={project 2},pdfkeywords={embedding theorems, left-ordered groups, simple groups, word problem, computable algebra},pdftex]{hyperref}

\hypersetup{colorlinks,%
citecolor=black,%
filecolor=black,%
linkcolor=black,%
urlcolor=black,%
}

\theoremstyle{plain}

    \newtheorem*{conclusion*}{Conclusion}
				\newtheorem*{nttn*}{Notation}
    \newtheorem*{acknowledgement}{Acknowledgement}
    
                    \newtheorem{propty}{Property}[section]
                    \newtheorem{df}{Definition}[section]

				\newtheorem{pbm}{Open problem}[section]
				\newtheorem{thm}{Theorem}[section]
				\newtheorem{props}{Proposition}[section]
                    \newtheorem{claim}{Claim}[section]

				\newtheorem{lem}[thm]{Lemma}		
						\newtheorem{cor}[thm]{Corollary}

\newtheorem{rem}[thm]{Remark}

	\theoremstyle{remark} 	

\DeclareMathOperator{\supp}{supp}

\DeclareMathOperator{\Wr}{\wr\wr}

\DeclareMathOperator{\zz}{\mathbb{Z}}
\DeclareMathOperator{\nn}{\mathbb{N}}

\DeclareMathOperator{\calZ}{\mathcal{Z}}
\DeclareMathOperator{\calA}{\mathcal{A}}
\DeclareMathOperator{\calB}{\mathcal{B}}

\DeclareMathOperator{\calD}{\mathcal{D}}
\DeclareMathOperator{\calT}{\mathcal{T}}
\DeclareMathOperator{\calH}{\mathcal{H}}

\DeclareMathOperator{\calF}{\mathcal{F}}
\DeclareMathOperator{\calN}{\mathcal{N}}

\DeclareMathOperator{\calP}{\mathcal{P}}
\DeclareMathOperator{\calQ}{\mathcal{Q}}
\DeclareMathOperator{\calR}{\mathcal{R}}
\DeclareMathOperator{\calM}{\mathcal{M}}

\DeclareMathOperator{\rf}{{RF}}

\DeclareMathOperator{\lcm}{lcm}

\newcommand{\red}[1]{\textcolor{red}{#1}}

\newcommand{\bft}[1]{\textbf{#1}} 
 
  {\begin{center}
    \colorbox{#1}{\begin{minipage}{0.9\textwidth}}}
  {\end{minipage}\end{center}}

\title{ A robust family of residually finite groups; spectra of residual finiteness growth, computability properties, and other applications}
\date{\today}
\author{Arman Darbinyan}
\begin{document}
\maketitle 
\[
\large{\textit{with an Appendix by Arman Darbinyan and Emmanuel Rauzy}}
\]
~\\
\begin{abstract}
In this paper, we introduce a family of residually finite groups that helps us to systematically study the residual finiteness  growth function (RFG) from various perspectives.

First, by strengthening results of Bou-Rabee and Seward and also of Bradford, we show that any non-decreasing function $f: \nn \rightarrow \nn$ that satisfies $f(n) > \exp{(\varepsilon n\log{n})}$ for some $\varepsilon>0$  can be realized (up to the standard equivalence) as RFG function of a two-generated residually finite group. Moreover, such a group can be found among solvable groups  of derived length $3$; due to what, in a strengthened way, we extend a theorem of Kharlampovich, Miasnikov and Sapir.

{Next, we consider computability aspects related to those growth functions. In particular, we characterize the decidability of the word problem in residually finite groups with respect to \emph{individual residual finiteness depth} functions. Then, we give a full description of sufficiently fast {growing} functions that are realizable as RFG for some group \emph{with decidable word problem} in terms of \emph{left-computable functions.} We also show that a Turing degree can be realized via RFG of a group with decidable word problem if and only if it is recursively enumerable.

Finally, applying the introduced theoretical framework, we answer several open questions and extend known results. For example, answering a question of Minasyan, by providing a construction, we show the existence of conjugacy separable groups with decidable word problem and undecidable conjugacy problem. Two more applications, including an answer to a question by Nies, can be found in the appendix coauthored with Rauzy.}



\end{abstract}

 \section{Introduction}

\subsection{Main results}

Let $G$ be a residually finite group generated by a finite set $X$ and let $g\in G\setminus \{1\}$. By $\rf_g \in \nn$ we denote the smallest positive integer for which there exists a normal subgroup $N_g \triangleleft G$ such that $[G:N_g]=\rf_g$ and $g \notin N_g$. We call $\rf_g$ \emph{the residual finiteness depth of $g$.} The concept of \emph{residual finiteness growth} was introduced by Bou-Rabee in \cite{bou_rabee-RFG} as a quantization of residual finiteness of groups. If we denote by $|\cdot|_X$ the length of elements of $G$ with respect to the finite generating set $X$, then the residual finiteness growth of $G$ (or, briefly, \emph{RFG} of $G$), denoted as $\rf_{G, X}: \nn \rightarrow \nn$, can be defined as $\rf_{G, X}(n) = \max \{\rf_g \mid g\in G\setminus \{1\}, |g|_X \leq n \}$. In cases when it does not lead to ambiguities, $\rf_{G, X}$ might be denoted by $\rf_{G}$ or by $\rf$. In \cite{bou_rabee-seward}, Bou-Rabee and Seward show that for any function $f:\nn \rightarrow \nn$ there exists a residually finite group the residual finiteness growth of which (asymptotically) dominates $f$. In \cite{bradford-RFG_spectrum}, Bradford recently strengthened this result from \cite{bou_rabee-seward} showing that for any large enough and sufficiently fast growing function $f$ one can find a residually finite group whose RFG is asymptotically equivalent to $f$. In \cite{kharlampovich_miasnikov_sapir}, Kharlampovich, Miasnikov and Sapir show that for any computable function $f$, one can find a finitely presented residually finite solvable group of derived length $3$ for which RFG dominates $f$. One of our main results is a strengthening of Bradford's result, by giving a complete description of large enough RFG functions up to a standard equivalence relation defined below.

To state the next theorem, we use a standard equivalence notion that makes the RFG independent of the choice of a finite generating set. That is, we say $f_1 \preceq f_2$ if there exists a constant $C$ such that $f_1(n) \preceq Cf_2(Cn)$. The equivalence $f_1 \simeq f_2$ means $f_1 \preceq f_2$ and $f_2 \preceq f_1$. It is an easy observation that up to this equivalence, RFG is independent of the choice of generating set. This notion is used throughout the paper.
\begin{thm}[Theorem \ref{thm-RFG-spectra-COPY}]
\label{thm-RFG-spectra}
    For any non-decreasing $f:\nn \rightarrow \nn$ such that for some $\varepsilon>0$, $f(n) \geq \exp{(\varepsilon n \log{n})}$ there exists a residually finite group $G_f$ such that $\rf_{G_f}(n) \simeq f(n).$ Moreover, $G_f$ can be found among  two-generated solvable groups of derived length $3$.
\end{thm}
\begin{rem}
     Theorem \ref{thm-RFG-spectra}  can also be regarded as an extension of the aforementioned result from \cite{kharlampovich_miasnikov_sapir} to the class of all finitely generated residually finite solvable groups of derived length $3$. 
\end{rem}
\begin{rem}
    Note that there is no hope to show Theorem \ref{thm-RFG-spectra} in the class of solvable groups of derived length $2$, as it is a well-known fact that there are only countably many finitely generated metabelian groups up to isomorphism. Therefore, the estimate $3$ of derived length is optimal in the theorem.
\end{rem}
\begin{rem}
    Theorem \ref{thm-RFG-spectra} strengthens the main theorem from \cite{bradford-RFG_spectrum}, where a similar result is shown for functions $f$ satisfying $f(n)\geq \exp(cn \log(n)^2 \log \log(n+\varepsilon))$ and, additionally, $f(Kn)>f(n)^K$ for some $c, \varepsilon>0$ and $K>1$.
\end{rem}
\subsubsection{The Word Problem and the Individual Residual Finiteness Depth function}

Let $G=\langle X \rangle$ be a residually finite group together with its finite generating set $X$. Let us define
\begin{align*}
    \rho = \rho_{G, X}: (X \cup X^{-1})^* \rightarrow \nn
\end{align*}
as follows: for $w \in (X \cup X^{-1})^*$, $\rho(w)=1$ if and only if $w=_G 1$. Otherwise, $\rho(w)$ is the smallest possible size of the finite quotient of $G$ in which the element represented by $w$ is not trivial. We call $ \rho: (X \cup X^{-1})^* \rightarrow \nn$ the \textbf{individual residual finiteness depth function} of $G$ with respect to $X$.

\begin{df}
    \label{def-left-partial-recursive}
    Let $f: S \rightarrow \nn$, where $S$ is a recursive (countable) set. Then, we say that $f$ is \textbf{left-computable } if the set 
    \[
    \{ (s,k)\in S \times \nn \mid f(s)>k \}.
    \]
    is recursively enumerable. Or, in other words, it means that there is a (partial) Turing machine that verifies the property $f(s)>k $ for $(s,k)\in S \times \nn$.
\end{df}

The next theorem is a characterization of the decidability of the word problem in a finitely generated residually finite group by its individual residual finiteness depth function.

\begin{thm}[Theorem \ref{thm-characterization-WP-decidability-COPY}]
    \label{thm-characterization-WP-decidability}
Let $G=\langle X \rangle$ be a residually finite with a recursively enumerable presentation along with its finite generating set $X$. Then the word problem in $G$ is decidable if and only if the individual residual finiteness depth function of $G$ is left-computable .
\end{thm}

An immediate corollary from Theorem \ref{thm-characterization-WP-decidability} is the following.
\begin{cor}
\label{cor-RFG-when_decidable_WP}
    If $G$ is a finitely generated residually finite group with decidable word problem, then the residual finiteness growth function $\rf: \nn \rightarrow \nn$ is left-computable .
\end{cor}
\begin{proof}
    Indeed, for $(n, k) \in \nn \times \nn$, $\rf(n)>k$ means that at least for one $w \in (X \cup X^{-1})^*$, where $X$ is a finite generating set of $G$, $\rho(w)>k$, which, by Theorem \ref{thm-characterization-WP-decidability} is algorithmically detectable. Therefore, $\rf: \nn \rightarrow \nn$ is left-computable .
\end{proof}
\subsubsection{The RFG functions for RF groups that have decidable WP}

Let $f: \nn \rightarrow \nn$ be such that 
\begin{itemize}
\item $f$ is left-computable,
    \item $f(n)$ is non-decreasing and $f(n)>\exp(n\log(n))$.
\end{itemize}
~\\
The next theorem can be regarded as a complement to Corollary \ref{cor-RFG-when_decidable_WP}.
\begin{thm}[Theorem \ref{thm-on-RFG-for-WP-COPY}]
    \label{thm-on-RFG-for-WP}
    For any $f: \nn \rightarrow \nn$ satisfying the above conditions, there exists a residually finite group $G_f$ with decidable word problem such that $f(n) \preceq \rf_{G_f}(n) \preceq (f(n))^6$  with respect to some generating set.
\end{thm}
\begin{rem}
    If $f: \nn \rightarrow \nn$ is such that  $f(n)>\exp(n\log(n))$ and for some constants $c_1, c_2 >1$, $f(n)^{c_1} \leq f(c_2 n)$, then Theorem \ref{thm-on-RFG-for-WP} implies that $f$ up to the standard equivalence is realizable as a RFG function for some residually finite group with decidable word problem if and only if it is equivalent to a non-decreasing left-computable function. A spectrum of RFG for functions satisfying the constraint $f(n)^{c_1} \leq f(c_2 n)$ is described in \cite{bradford-RFG_spectrum}.

\end{rem}

\subsubsection{Obstructions for Higman's embedding theorem for RF groups with decidable WP, and a generalization}
The most common way to measure the undecidability of a given algorithmic problem is via Turing degrees, which is the central object of study in the computability theory. Given two decision problems $\calD_1$ and $\calD_2$, we say that $\calD_2$ has at least the same Turing degree as $\calD_1$ if given an oracle that solves the problem $\calD_2$, one can algorithmically (i.e. using a Turing machine) solve the problem $\calD_1$. It is denoted as $\calD_1 \leq_T \calD_2$. If $\calD_1 \leq_T \calD_2$ and $\calD_1 \leq_T \calD_2$, then it is said that $\calD_1$ and $\calD_2$ are of the same Turing degree  or they are Turing equivalent. Thus, one can consider equivalence classes of Turing degrees. For example, the class of problems that can be solved by a Turing machine (i.e., by a computer), is denoted by $0$, the class of problems that are equivalent to Turing's halting problem is denoted by $0'$. A basic fact is that for any recursively enumerated subset $\calN \subset \nn$, the Turing degree of the membership problem for $\calN \subset \nn$, $\sigma_{\calN}$, satisfies $0 \leq_T \sigma_{\calN} \leq_T 0'$ (though the inverse statement is not true \cite{soare-computability}). Such a Turing degree is called \emph{recursively enumerable}. It is well-known that there are countably many pairwise different recursively enumerable Turing degrees. Similarly, we say that a map $f: \nn \rightarrow \nn$ is of Turing degree $\sigma$ if the problem of computing the values of $f$ is Turing equivalent to the problems in the Turing class $\sigma$. For a comprehensive discussion of the theory of Turing degrees, see \cite{soare-computability}. 

As a direct corollary from Theorem \ref{thm-characterization-WP-decidability}, we obtain the following.
\begin{cor}
    \label{cor-from-individual-rfg}
    Let $G = \langle X \rangle$ be a finitely generated residually finite group with decidable word problem. Then, $\rf_{G, X}(n)$ is of recursively enumerable Turing degree.
\end{cor}
\begin{proof}
    Note that if $G$ has decidable word problem, then knowing how to compute the individual residual depth function $\rho_{G, X}: (X \cup X^{-1})^* \rightarrow \nn$, one can compute the residual finiteness growth function $\rf_{G, X}: \nn \rightarrow \nn$. In other words, $\rho_{G, X}$ can serve as an oracle for computing $\rf_{G, X}$. Thus it is enough to show that $\rho_{G, X}$ is of recursively enumerable Turing degree. To show the latter, one only has to notice that given an oracle for the membership problem for $\{ (w, n) \mid w \in (X \cup X^{-1})^*, \rho_{G, X}(w)>n \} \subset (X \cup X^{-1})^*$, one can algorithmically compute $\rho_{G, X}$. Finally, by Theorem  \ref{thm-characterization-WP-decidability}, $\{ (w, n) \mid w \in (X \cup X^{-1})^*, \rho_{G, X}(w)>n \}$ is recursively enumerable. Thus, the corollary is proved.
\end{proof}
\begin{df}
     Let $\sigma$ be a Turing degree and $f: \nn \rightarrow \nn$. We say that $f$ is \bft{$\sigma$-dominant} if its Turing degree is $\sigma$ and any $g: \nn \rightarrow \nn$ such that $g(n) \geq f(n)$ has Turing degree at least $\sigma.$
 \end{df}

In Subsection \ref{subsec-turing-dominant}, as a consequence from Theorem \ref{thm-characterization-WP-decidability}, we get the following.
\begin{thm}[Theorem \ref{thm-turing-dominant-COPY}]
\label{thm-turing-dominant}
   For any r.e. Turing degree $\sigma$ there exists a RF group $G_{\sigma}$ with decidable word problem and {$\sigma$-dominant} RFG function up to the standard equivalence. 
\end{thm}
\begin{rem}
    Note that {$\sigma$-dominanteness} of RFG function with respect to a fixed generating set implies it is {$\sigma$-dominant} with respect to all finite generating sets, which, in particular, contrasts the groups $G_{\sigma}$ with those described in Theorem \ref{thm-answer-to-Rauzy-1}.
\end{rem}
\begin{rem}
    Theorem \ref{thm-turing-dominant} in particular implies that for any recursively enumerable Turing degree $\sigma$, there is a finitely generated residually finite group for which the RFG is of degree $\sigma$. Corollary \ref{cor-from-individual-rfg} states that the inverse is also true, that is, every finitely generated residually finite group with decidable word problem has RFG of recursively enumerable Turing degree. Therefore, Theorem \ref{thm-turing-dominant}, in particular, describes the full residual finiteness growth spectra for groups with decidable word problem in terms of Turing degrees.
\end{rem}
It is well-known that finitely presented residually finite groups have decidable word problem \cite{McKinsey_1943}. An immediate application of this fact is that a finitely generated residually finite group with undecidable word problem cannot embed into a finitely presented residually finite group, which means that the direct analogue of the famous embedding theorem of Higman cannot be restricted to the class of residually finite groups. Given this observation, Kharlampovich, Miasnikov and Sapir ask in \cite{kharlampovich_miasnikov_sapir} if the undecidability of the word problem is the only obstruction for the existence of such embeddings. See also \cite{Steinberg-s_question} for the same question asked by Steinberg. In \cite{rauzy-on_higman_embedding}, Rauzy showed that the answer to that question is negative by constructing a residually finite groups with decidable word problem that cannot be embedded into a finitely presented residually finite group. In fact, Rauzy's group cannot be embedded into any finitely presented residually finite group with decidable word problem. The next corollary is a natural computability theoretical extension of Rauzy's theorem.
\begin{cor}
\label{cor-123451}
    For any recursively enumerable Turing degree $\sigma$, there exists  a residually finite two-generated solvable group $G_{\sigma}$ of derived length $3$ such that it has decidable word problem and the Turing degree of $\rf_{G_{\sigma}}(n)$ is $\sigma$ and which is not a subgroup of a finitely generated residually finite group that has RFG of Turing degree less than $\sigma$. In particular, if $0 <_T \sigma$, then $G_{\sigma}$ does not embed into a finitely presented residually finite group.
\end{cor}
\begin{proof}
    As noted by Bou-Rabee and Seward \cite{bou_rabee-seward}, if $G \leq H$ are finitely generated residually finite groups, then $\rf_G(n) \preceq \rf_H(n)$. Because of this property, to prove Corollary \ref{cor-123451}, it is enough to consider the group $G_{\sigma}$ from Theorem \ref{thm-turing-dominant}.
\end{proof}

As noted by Rauzy in \cite{Rauzy_quotients}, if finitely generated residually finite group has \emph{Computable Finite Quotients property} (for definition, see \cite{Rauzy_quotients}), then the computability of the RFG of the group is independent of the choice of the generating set. Moreover, in this case the decidability of the word problem is equivalent to the computability of RFG, \cite[Corollary 26]{ Rauzy_quotients}. A natural question, asked by Rauzy (see \cite[Problem 25]{Rauzy_quotients}; see also \cite[Problem 3.3.3.]{rauzy_thesis}), therefore asks if there is a finitely generated residually finite group $G$ with two finite generating sets $S$ and $S'$ such that $\rf_{G, S}$ is computable but $\rf_{G, S'}$ is not? 

The next theorem answers affirmatively to this question.
\begin{thm}[Theorem \ref{thm-answer-to-Rauzy-1-COPY}]
    \label{thm-answer-to-Rauzy-1} 
    There exists a residually finite two-generated solvable group $G_{c}$ of derived length $3$ such that the computability property of its residual finiteness growth function depends on the choice of the generating set. 
\end{thm}

\subsubsection{Conjugacy separable groups with decidable word problem and undecidable conjugacy problem: an application of the main construction}
A well-known theorem of Mal'cev states that for finitely presented conjugacy separable group, the conjugacy problem is decidable. A natural question in this context would be a question asked by Minasyan (see \cite{minasyan-s_question}) about the existence of a finitely generated conjugacy-separable group with decidable word problem but undecidable conjugacy problem. The theorem stated below shows that such groups exist. The proof, which is given in Section \ref{sec-ashot-s_question}, is constructive and is based on the framework introduced in this work. 
\begin{thm}[Theorem \ref{thm-ashot-s_question-COPY}]
\label{thm-ashot-s_question}
There exists a finitely generated conjugacy separable group with a decidable word problem and an undecidable conjugacy problem.
\end{thm}
\begin{rem}
    The proof of Theorem \ref{thm-ashot-s_question}, in addition, shows that such groups can be found among two-generated solvable groups of derived length $3$. Other paradoxical properties of the word and conjugacy problems in solvable groups of derived length $3$ can be found, for example, in \cite{darbinyan_inventiones}.
\end{rem}

 \subsection{Plan of the paper}
In Section \ref{section-main-construction}, we introduce the main construction for the family of two-generated solvable groups $G_{\calA}$ of derived length $3$ that are considered throughout this paper, where $\calA$ stands for a countable abelian group, then describe some of its basic properties. In Section \ref{section-presentation_of_G_A}, we describe an explicit presentation of that family of groups and consider some further properties. In theory, we could skip Section \ref{section-main-construction} and define the groups via their presentation as in \ref{section-presentation_of_G_A}; however, this approach allows viewing the groups under discussion from different viewpoints and gives them more versatility for further modifications. In Section \ref{sec-word_problem}, we describe a necessary and sufficient condition on $\calA$ that makes $G_{\calA}$ a group with decidable word problem. In Section \ref{sec-metric_properties}, we discuss some basic word-metric properties of the groups $G_{\calA}$ that we repetitively use in applications. In Section \ref{section-periodicity-and-RF}, we describe a condition, which we call \emph{periodicity}, on the countable groups $\calA$ that make the family of groups $G_{\calA}$ residually finite. In addition, in Section \ref{section-periodicity-and-RF} we also introduce new terminology and a normal form for groups $G_{\calA}$ that we use in the main applications. In Sections \ref{section-RFG-robustness}-\ref{sec-ashot-s_question} we describe the proofs of the main theorems that are all based on different adaptations of the groups $G_{\calA}$; namely, for each of them we describe particular presentations for $\calA$. It is worth mentioning that Sections \ref{section-RFG-robustness}-\ref{sec-ashot-s_question} can be read independently of each other. An exception is the proof of Theorem \ref{thm-answer-to-Rauzy-1} in Subsection \ref{subsection-Rauzy_1}, which is based on the discussion in Section \ref{section-robustness-torsion_case}.

\subsection{Notations and conventions} Throughout this paper, we denote $[a, b]=aba^{-1}b^{-1}$, $a^b=bab^{-1}$. For groups $A, B$, by $A \Wr B$ and $A \wr B$ we denote the unrestricted and restricted wreath products of $A$ and $B$, respectively. More precisely, we define $A \Wr B$  as the semidirect product of $A^{B}:=\{f: B \rightarrow A\} $, given with pointwise multiplication, and $B$ with respect to the action $(b' \cdot f)(b) = f(bb')$ for $b', b \in B, f\in A^B$. For  $A \wr B$, the maps $\{f: B \rightarrow A\}$ are with finite support. 

For a finitely generated group $G = \langle X \rangle$, by $|g|_X$ we denote the length of the element $g \in G$ with respect to the generating elements $X \cup X^{-1}$. Assuming that there is on ambiguity, we simply might denote $|g|$ instead of $|g|_X$, which is the length of a shortest word $w\in (X \cup X^{-1})^*$ such that $w=_G g$. The groups $G_{\calA}$ always come with a generating set $\{F, s \}$, which we call the \emph{standard set of generators} for $G_{\calA}$.

 \section{Definition of the groups $G_{\mathcal{A}}$}

\label{section-main-construction}

Let $$\mathcal{A} = \langle a_i, i \in \zz \mid [a_i, a_j]=1, i, j \in \zz; a_{2i}=1, i \in \zz;  \calR  \rangle$$ be an abelian group with generators $\{a_{\pm 1}, a_{\pm 2}, \ldots\}$, where $\calR$ is some set of additional relators, which will be specified in applications in later sections.
\begin{rem}
    In the current paper, in what follows in the main discussion, we will always assume that $a_n=a_{-n}$ for all $n\in \zz$. However, for expository purpose, for now we skip this assumption. See Remark \ref{rem-another-condition-on-a_n}. Note that instead of renumbering the generating elements of $\calA$, we add the condition $a_{2i}=1$, which is also done for expository reasons. The convenience of the latter notation will become evident later, when we start considering $\calA$ as a subgroup of another abelian group $\calB$.
\end{rem}

Let \[\calZ = \langle z_i, i \in \zz \mid [z_i, z_j]=1 \text{ for } i, j \in \mathbb{Z} \rangle\] be the free abelian group with the countable basis $\{z_i \mid i \in \mathbb{Z} \}$.

We start by introducing  the following notations and structures:
\begin{itemize}
    \item Let $z \in \calZ$. Let $I$ be a finite subset of $\mathbb{N}$ such that $z=\prod_{i \in I} z_{i}^{k_i}$ for $k_i \in \mathbb{Z}\setminus \{0\}$. Then we call it a \emph{canonical decomposition} of $z$. Also, we define $supp(z):=I$ and, for $i\in \mathbb{Z}$, define 
    \[|z|_{z_i} =  \left\{
                      \begin{array}{ll}
                       k_i  & \mbox{if $i \in I$ ,}\\
                       0  & \mbox{otherwise.}
                       \end{array}
                    \right. 
\]
    
    \item For $m, n \in \zz$ such that $m$ is even and $n$ is odd, define $$f_{m, n}: \calZ \rightarrow \langle a_n \rangle$$ as $$f_{m, n}(z) = a_n^{|z|_{z_m}} \text{~for all~} z \in \calZ.$$

    \item For odd $n \in \zz$, define $$\tilde f_n: \calZ \rightarrow \langle a_n \rangle$$ to be the constant function such that $$\tilde f_n(z) \equiv a_n, ~ z \in \calZ.$$

    \item Let $\calA^{\calZ}$ be the group of functions from $\calZ$ to $\calA$ endowed with \\coordinate-wise multiplication.

    \item For odd $n \in \zz$, define $\bar{f}_n: \calZ \rightarrow \calA$ as 
    $$ \bar{f}_n(z) = \prod_{i \in supp(z)} a_{n-i}^{|z|_{z_i}} \text{~for~} z \in \calZ.$$

    \item Let $\calH := \calA \Wr \calZ$ be the unrestricted wreath product of $\calA$ and $\calZ$. More precisely, we define $\calH$ as the semidirect product of $\calA^{\calZ}$ and $\calZ$ with respect to the action $(z' \cdot f)(z) = f(z'z)$ for $z', z \in \calZ, f\in \calA^{\calZ}$.

    \item Let $F: \langle s \rangle \rightarrow \calH$ be defined as 
    \[
    F(s^n) = \left\{
                      \begin{array}{ll}
                       \bar{f}_n  & \mbox{if $n\in \zz$ is odd ,}\\
                       z_n  & \mbox{if $n\in \zz$ is even,}
                       
                     \end{array}
                    \right. 
    \]
where $\langle s \rangle$ is an infinite cyclic group.
    \item Define $G_{\calA}$ to be the subgroup of the unrestricted wreath product $\calH \Wr \langle s \rangle$ generated by the elements $F$ and $s$. That is $$ G_{\calA} =\langle F, s \rangle < \calH \Wr \langle s \rangle.$$
   
\end{itemize}

\subsection{Core properties of $G_{\calA}$}
\begin{propty}
\label{property-1}
    For even $k, m \in \zz$ and odd $n \in \zz$ and for $p, q \in \zz$, we have
    $$ [z_k^p, \left(f_{m, n}\right)^q] =  \left\{
                      \begin{array}{ll}
                       1  & \mbox{if $k \neq m$ ,}\\
                       \left(\tilde{f}_n\right)^{pq}  & \mbox{if $k=m$.}
                       
                     \end{array}
                    \right. $$
\end{propty}
\begin{proof}
    Let $z \in \calZ$. Then, 
    $$ [z_k^p, f_{m, n}^q] (z) = \left(f_{m, n}^{z_k^p}\right)^q (z) f_{m, n}^{-q} (z) \\
    =  \left(f_{m, n} (zz_k^p)\right)^q (f_{m, n}(z))^{-q} = a_n^{q|zz_k^p|_{z_m}} a_n^{-q|z^p|_{z_m}} = a_n^{q(|zz_k^p|_{z_m}-|z|_{z_m})}.
     $$
     Now, note that if $k=m$, then $ |zz_k^p|_{z_m} - |z|_{z_m} =p$, otherwise $ |zz_k^p|_{z_m} - |z|_{z_m} =0$, hence the proof follows.
\end{proof}
\begin{propty}
\label{property-2}
    For even $m\in \zz$ and odd $n\in \zz$, we have $[z_m, \bar{f}_n] = \tilde{f}_{n-m}$ and $[z_m^{-1}, \bar{f}_n] = [z_m, \bar{f}_n^{-1}] = \left( \tilde{f}_{n-m}\right)^{-1}$. Also, we have $[z_m, \tilde f_n]=1$.
\end{propty}
\begin{proof}
    Let $z \in \calZ$. Note that, by definition, $$\bar f_n(z) =  \prod_{i \in supp(z)} a_{n-i}^{|z|_{z_i}} = \prod_{i \in supp(z)} f_{i, n-i}(z).$$
 Therefore, since elements $[z_k, f_{m, n}]$, $k, m, n \in \zz$, commute with each other, we get 
    \begin{align}
        \label{eq-1-1}
        [z_m, \bar f_n](z) &=  \bar f_{n}^{z_m}(z) \bar f_{ n}^{-1} (z) = \bar f_{n}(z z_m) \bar f_{ n}^{-1} (z)
        &= \left( \prod_{i \in supp(zz_m)} f_{i, n-i}(zz_m) \right) \left( \prod_{i \in supp(z)} f_{i, n-i}(z) \right)^{-1}.
    \end{align}

    Now, note that if $i \notin supp(z)$, then $f_{i, n-i}(z) = 1$. Similarly,  if $i \notin supp(zz_m)$, then $f_{i, n-i}(zz_m) = 1$. Also, note that $m \in supp(zz_m) \cup supp(z)$.
    Therefore, \eqref{eq-1-1} implies,
 \begin{align*}
        [z_m, \bar f_n](z) &=  \\
        &= \left( \prod_{i \in supp(zz_m) \cup supp(z)} f_{i, n-i}(zz_m) \right) \left( \prod_{i \in supp(zz_m) \cup supp(z)} f_{i, n-i}(z) \right)^{-1} \\
        &=  \prod_{i \in supp(zz_m) \cup supp(z)} f_{i, n-i}(zz_m)  (f_{i, n-i}(z) )^{-1}\\
        &=  \prod_{i \in supp(zz_m) \cup supp(z)} [z_m, f_{i, n-i}](z) = [z_m, f_{m, n-m}](z) = \tilde f_{n-m} (z).
    \end{align*}
    The last two equalities follow from Property \ref{property-1} (for $p=1$ and $q=1$).

    To prove $[z_m^{-1}, \bar{f}_n] = \left( \tilde{f}_{n-m}\right)^{-1}$, we repeat the above argument and note that by Property \ref{property-1} (for $p=-1$ and $q=1$), $ [z_m^{-1}, f_{m, n-m}](z) = \left( \tilde f_{n-m} (z) \right)^{-1}.$ 
    Similarly, we have $[z_m, \bar{f}_n^{-1}] = \left( \tilde{f}_{n-m}\right)^{-1}$.

    Finally, for the second assertion, note that $[z_m, \tilde f_n](z) = \tilde f_n(z z_m) \left( \tilde f_n(z)\right)^{-1} = 1$, because, by definition, $\tilde f_n: \calZ \rightarrow \calA$ is constant.
\end{proof}

 The next notation is central for the latter exposition.
\begin{nttn*}    
Denote $F_n: = [F, F^{s^n}] \in \calH^{\langle s \rangle}$.
\end{nttn*}

\begin{propty}
    
\label{property-3}
For $n, k \in \zz$, we have $$F_n(s^k) = \left\{
                      \begin{array}{ll}
                       \tilde{f}_{(-1)^kn}^{(-1)^k}    & \mbox{if $n $ is odd ,}\\
                       1  & \mbox{if $n$ is even.}
                       
                     \end{array}
                    \right. $$
\end{propty}
\begin{proof}
    For $k\in \zz$, we have 
     
     \begin{align*}
     F_n(s^k) &= (F F^{s^n} F^{-1} F^{-s^n})(s^k) = F(s^k) F(s^{n+k}) F^{-1}(s^k) F^{-1}(s^{n+k})\\ 
     &= \left\{
                      \begin{array}{ll}              
                       \bar{f}_k z_{n+k} \bar f^{-1}_k z^{-1}_{n+k}    & \mbox{if $n$ is odd and $k$ is odd,}\\
                       z_k \bar{f}_{n+k} z^{-1}_k \bar f^{-1}_{n+k}    & \mbox{if $n$ is odd and $k$ is even,}\\
                       \bar{f}_k \bar{f}_{n+k} \bar f^{-1}_k \bar f^{-1}_{n+k}    & \mbox{if $n$ is even and $k$ is odd,}\\
                       {z}_k {z}_{n+k} z^{-1}_k z^{-1}_{n+k}    & \mbox{if $n$ is even and $k$ is even.}\\
                         \end{array}
                    \right. \\
                       &= \left\{
                      \begin{array}{ll}
                       {[}\bar{f}_k, z_{n+k}]    & \mbox{if $n$ is odd and $k$ is odd,}\\
                       {[}z_k, \bar{f}_{n+k}]    & \mbox{if $n$ is odd and $k$ is even,}\\
                       1    & \mbox{if $n$ is even.}\\
                         \end{array}
                    \right. \\
    \text{by Property \ref{property-2},~~} 
    &= \left\{
                      \begin{array}{ll}
                       \tilde f_{-n}^{-1}    & \mbox{if $n$ is odd and $k$ is odd,}\\
                       \tilde f_{n}    & \mbox{if $n$ is odd and $k$ is even,}\\
                       1    & \mbox{if $n$ is even.}
                         \end{array}
                    \right. \\
    &= \left\{
                      \begin{array}{ll}
                       \tilde f_{(-1)^k n}^{(-1)^k}    & \mbox{if $n$ is odd,}\\
                       1    & \mbox{if $n$ is even.}
                         \end{array}
                    \right. 
        \end{align*}
\end{proof}
\begin{cor}
\label{cor-1}
  If for all odd $k\in \zz$ we have $a_{k} = a_{-k}$, then for all $n \in \zz$, $F_n^s = F_n^{-1}$ and $F_n=F_{-n}$. 
\end{cor}
\begin{proof}
    Indeed, if $n$ is even, then $F_n^s = 1 = F_n^{-1}$. Now assume that $n$ is odd. Then, by Property \ref{property-3}, we get, for all $k \in \zz$, $F_n^s(s^k) = F_n(s^{k+1}) = \tilde f_{(-1)^{k+1} n}^{(-1)^{k+1}} = \left(\tilde f_{-(-1)^{k} n}^{(-1)^k}\right)^{-1} = \left(\tilde f_{(-1)^{k} n}^{(-1)^k}\right)^{-1} = (F_n(s^k))^{-1}$. Finally, the condition $a_n=a_{-n}$ along with Property \ref{property-3}, immediately implies $F_n=F_{-n}$. Thus the corollary is proved. 
\end{proof}
\begin{rem}
\label{rem-another-condition-on-a_n}
    In case $a_{k} = a_{-k}^{-1}$, we get another nice property that the elements $F_n$, $n\in \zz$, generate the center of the group $G_{\calA}$. (In fact, in this case, the corresponding group $G_{\calA}$ would be a central quotient of the Hall groups \cite{hall-solvable}). However, in this case, we would sacrifice the property that $F_{n}=F_{-n}$, which is essential for the residual finiteness property that we obtain later, in Section \ref{section-periodicity-and-RF}.
\end{rem}
From now on, we assume that 
\begin{align*}
    a_n = a_{-n} \text{~for all~} n \in \zz.
\end{align*}
\begin{nttn*}
Denote $$\calF: = \langle F^{s^n} ;     ~n \in \zz \rangle \leq G_{\calA}.$$
 Also, denote   $$\calF': = \langle F_n ; ~n \in \zz \rangle = \langle F_n ; ~n \text{~is odd and positive~} \rangle \leq \calF \leq G_{\calA}.$$
\end{nttn*} 
\begin{rem}
    As Corollary \ref{cor-F'-derived-subgroup} indicates below, $\calF'$ is the derived subgroup of $\calF$.
\end{rem}

The next proposition, which immediately follows from Property \ref{property-3}, is of key importance for the applications.
\begin{props}
    \label{prop-about-calF'-calA-isom}
    $\calF'$ is an abelian group and the map $F_n \mapsto a_{n}$ induces an isomorphism between $\calF'$ and $\calA$.
\end{props}
\begin{proof}
    As Property \ref{property-3} indicates, we have $F_n(s^k)$ is equal to $ \tilde{f}_{(-1)^kn}^{(-1)^k}$    \mbox{if $n $ is odd ,} and to $1$  otherwise.
    Now, to finish the proof, we just need to recall that in $\calA$, we assume $a_n=a_{-n}$, $a_{2n}=1$, $n\in \zz$, and that $\tilde{f}_n: \calZ \rightarrow \calA$ is the constant map taking the value $a_n\in \calA$ endowed with the coordinatewise multiplication with other maps from $\calZ$ to $\calA$.
\end{proof}
\begin{props}
\label{prop-key-identities}
    For $m, n \in \zz$, $$[F_n, F^{s^m}]=1$$ and for $k, l \in \zz$, $$[F^k, (F^{s^n})^l]=F_n^{kl}.$$
More generally, we have 
    $$[(F^{s^m})^k, (F^{s^n})^l]=F_{|m-n|}^{(-1)^m kl}.$$
\end{props}
\begin{proof}
    First, we will prove that for any $m, n \in \zz$, $[F_n, F^{s^m}]=1$. For even $n$ this is obvious, as $F_n=1$. Therefore, assume $n$ is odd. Then, for any $r\in \zz$, we have $[F_n, F^{s^m}](s^r) = [F_n(s^r), F(s^{m+r})]=1$, where the last equality follows from the commutativity of $\calZ$, and Properties \ref{property-2} and  \ref{property-3}, because $F_n(s^r) \in \langle \tilde f_i; i \in \zz \rangle$ and $ F(s^{m+r}) \in \{z_i, \bar f_i \mid i \in \zz \} $. 

    To prove the second assertion, first let us show that $[F^{-1}, F^{s^n}] = F_n^{-1}$. Indeed, for $r \in \zz$, we have 
    \begin{align}
        \label{eq-interim-1}
        [F^{-1}, F^{s^n}](s^r) = F^{-1}(s^r) F(s^{n+r})F(s^r)F^{-1}(s^{n+r}).
    \end{align}
    For $n$ even, we immediately get $[F^{-1}, F^{s^n}](s^r)=1$. For $n$ is even, we will separately consider the cases of odd and even $r$.\\

    Case 1 ($r$ is even). In this case, \eqref{eq-interim-1} becomes equal to $[z_r^{-1}, \bar f_{n+r}]$ which is equal to $\tilde f_n^{-1}$ by Property \ref{property-2}.

    Case 2 ($r$ is odd). In this case, \eqref{eq-interim-1} becomes equal to $[\bar f_r^{-1}, z_{n+r}] = [z_{n+r}, \bar f_r^{-1}]^{-1}$, which is equal to $\tilde f_{-n}$ by Property \ref{property-2}. 

    Combining Cases 1 and 2, we get $$[F^{-1}, F^{s^n}](s^r)= \left\{
                      \begin{array}{ll}
                       \tilde{f}_{(-1)^rn}^{(-1)^{r+1}}    & \mbox{if $n \in \zz$ is odd ,}\\
                       1  & \mbox{if $n$ is even.}
                       \end{array}
                    \right.
    $$
    which means that $[F^{-1}, F^{s^n}] = F_n^{-1}$ by Property \ref{property-3}.
    
    Now, in a similar way, let us show that $[F, \left(F^{s^n}\right)^{-1}] = F_n^{-1}.$ Indeed, for $r \in \zz$, we have 
    \begin{align}
        \label{eq-interim-2}
        [F, \left(F^{s^n}\right)^{-1}](s^r) = F(s^r) F^{-1}(s^{n+r})F^{-1}(s^r)F(s^{n+r}).
    \end{align}
    For $n$ even, we get $1$. For odd $n$, as above, we will consider separately the cases of odd and even $r$.

    Case 1 ($r$ is even). In this case, \eqref{eq-interim-2} becomes equal to $[z_r, \bar f_{n+r}^{-1}]$ which is equal to $\tilde f_n^{-1}$ by Property \ref{property-2}.

    Case 2 ($r$ is odd). In this case, \eqref{eq-interim-2} becomes equal to $[\bar f_r,  z_{n+r}^{-1}] = [z_{n+r}^{-1}, \bar f_r^{-1}]^{-1}$, which is equal to $\tilde f_{-n}$ by Property \ref{property-2}. Thus, again we get
    $$[F, \left(F^{s^n}\right)^{-1}] = \left\{
                      \begin{array}{ll}
                       \tilde{f}_{(-1)^rn}^{(-1)^{r+1}}    & \mbox{if $n \in \zz$ is odd ,}\\
                       1  & \mbox{if $n$ is even.}
                       \end{array}
                    \right.
    $$
    which implies that $[F, \left(F^{s^n}\right)^{-1}] = F_n^{-1}$ by Property \ref{property-2}.

    Now, to complete the proof of the second assertion, we just need to combine these two observations with the commutator identities $[ab, c] = [b, c]^a [a, c]$ and $[a, bc] = [a, c][a, b]^c$, and with the first identity in the assertion of the proposition.

    Finally, to show the last identity from the statement of the proposition, we notice that $[(F^{s^m})^k, (F^{s^n})^l] = [F^k, (F^{s^n-m})^l]^{s^m}$, then apply the second identity from the proposition's statement along with the identities $F_n=F_{-n}$ and $F_n^{s}=F_n^{-1}.$
\end{proof}
A straightforward consequence from Proposition \ref{prop-key-identities} is the following.
\begin{cor}
\label{cor-F'-derived-subgroup}
    The group $\calF'$ defined above is the first derived subgroup of $\calF$. That is, $\calF' = [\calF, \calF]$.
\end{cor}
\begin{cor}
    $\calF'$ is a normal subgroup of $G_{\calA}$.
\end{cor}
\mbox{}

\section{The reduced form and an explicit presentation of $G_{\calA}$}
\label{section-presentation_of_G_A}
In this section, we give an explicit presentation of $G_{\calA}$ and characterize when $G_{\calA}$ has a recursive presentation and when has decidable word problem.\\

As previously, for $n\in \zz$, we denote
\[
F_n:=[F, F^{s^n}]: = FF^{s^n}F^{-1} (F^{s^n})^{-1}.
\]
\begin{df}[The swap operation]
\label{def-swap_operation}
Let $g \in G_{\calA}$ be given in a decomposition
 \begin{align*}
        g = s^x\left(  \prod_{i=1}^{t} (F^{s^{n_i}})^{k_i} \right).
    \end{align*}
 (Note that every element of $G_{\calA}$ can be freely represented by such decompositions. Moreover, $x \in \zz$ is uniquely defined by $g$.) Then we define the \textbf{swap operation} on the above decomposition as the transition from that decomposition to the decomposition 
 \begin{align*}
     g= s^x\left(  \prod_{i=1}^{j-1} (F^{s^{n_i}})^{k_i} \right) \left((F^{s^{n_{j+1}}})^{k_{j+1}}  (F^{s^{n_{j}}})^{k_{j}}\right)\left(  \prod_{i=j+2}^{t} (F^{s^{n_i}})^{k_i} \right) F_{|n_{j+1}-n_j|}^{(-1)^{n_j}k_jk_{j+1}},
 \end{align*}
 where $1\leq j<t.$ In other words, the swap operation is the application of the identity 
 \[
(F^{s^{n_{j}}})^{k_{j}}  (F^{s^{n_{j+1}}})^{k_{j+1}} = (F^{s^{n_{j+1}}})^{k_{j+1}}  (F^{s^{n_{j}}})^{k_{j}}F_{|n_{j+1}-n_j|}^{(-1)^{n_j}k_jk_{j+1}},
 \]
 and commutativity of $F_{|n_{j+1}-n_j|}^{(-1)^{n_j}k_jk_{j+1}}$  with factors of the form $F^{s^{k}}$. The observation that swap operations still decompose the same element $g\in G_{\calA}$
 is a direct consequence of Proposition \ref{prop-key-identities}. 
 \end{df}
\begin{df}[The reduced form]
   Let $g \in G_{\calA}$. The following presentation we call a \textbf{reduced form}:
    \begin{align}
    \label{eq-reduced-form}
        g = s^x\left(  \prod_{i=1}^{t} (F^{s^{n_i}})^{k_i} \right)f ,
    \end{align}
where $f \in \calF'$, $x \in \zz $, $n_1, \ldots, n_t$ are integers such that $n_1>n_2>\ldots>n_t>0$, $k_1, \ldots, k_t$ are non-zero integers, $t \in \nn_0$, where $t=0$ means that $g= s^x f$.
\end{df}
\begin{props} [Existence and uniqueness of the reduced from]
\label{prop-for-reduced-form-uniqueness}
\label{prop-on-reduced-form}
    Every element of $G_{\calA}$ can be decomposed in a reduced form as in \ref{eq-reduced-form}. Moreover, the parameters in \ref{eq-reduced-form}, namely  $x$, $n_1, \ldots, n_t$ and $k_1, \ldots, k_t$, as well as $f\in \calF'$, are uniquely defined by the element $g\in G_{\calA}$.
\end{props}
\begin{proof}
    The existence of a reduced form decomposition for every $g \in G_{\calA}$ follows from the fact that $g$ can be decomposed as
    \begin{align}
    \label{eq-345}
        g = s^y\left(  \prod_{i=1}^{t'} (F^{s^{m_i}})^{l_i} \right),
    \end{align}
    (which is a generic statement, not depending on the properties of $G_{\calA}$ besides that it is generated by $\{s, F \}$.) and the observation that the iterative application of the swap operations from Definition \ref{def-swap_operation} on \eqref{eq-345} can transform it to a reduced form decomposition, while guaranteeing that the new decomposition represents the same element $g$.

Now, let us show the uniqueness part of the proposition.
    The uniqueness of $x \in \zz$ is obvious and follows from the definition of $G_{\calA}.$ To show the uniqueness of the parameters $n_1, \ldots, n_t$ and $k_1, \ldots, k_t$, notice that $\left(  \prod_{i=1}^{t} (F^{s^{n_i}})^{k_i} \right)f$, when regarded as a map from $\langle s \rangle$ to $\calH$ (see Section \ref{section-main-construction} for the notation), for even $n\in \nn$ maps $s^n$ to \[
    \prod_{n_i \text{is even}} z_{n_i+n}^{k_i}\prod_{n_i \text{is odd}} \bar f_{n_i+n}^{k_i}f(s^n) \in \calH=\calA \Wr \calZ,
    \]
    and for odd $n\in \nn$, maps $s^n$ to 
    \[
    \prod_{n_i \text{is odd}} z_{n_i+n}^{k_i}\prod_{n_i \text{is even}} \bar f_{n_i+n}^{k_i} f(s^n) \in \calH=\calA \Wr \calZ.
    \]
    Since, by construction, $\prod_{n_i \text{is odd}} \bar f_{n_i+n}^{k_i}f(s^n) \in \calA^{\calZ}$ and $\{z_1, z_2, \ldots \}$ forms a basis of the free abelian group $\calZ$ (the latter implies that the decompositions $ \prod_{n_i \text{is even}} z_{n_i+n}^{k_i}$ and $ \prod_{n_i \text{is odd}} z_{n_i+n}^{k_i}$ are uniquely defined), we immediately obtain the uniqueness of  $n_1, \ldots, n_t$ and $k_1, \ldots, k_t$. Finally, the latter implies the uniqueness of $f\in \calF'$.
    
\end{proof}
\begin{rem}
\label{rem-for-reduced-form}
    It is evident from the proof of Proposition \ref{prop-for-reduced-form-uniqueness} that for any element $g\in G_{\calA}$, given as a word in generators $\{F, s\}$ of $G_{\calA}$, one can algorithmically compute the reduced form of $g$. Moreover, one can algorithmically compute $f$ in \eqref{eq-reduced-form} in a form of a product of factors given as $F_n$, $n\in \zz$. Also, as the process of obtaining the reduced form involves only the iterative application of swap operations, we conclude that the existence of reduced forms is a consequence of identities from Proposition \ref{prop-key-identities}. Moreover, as an outcome of the algorithmic process of applying the swap operations for obtaining the reduced form, one gets the element $f\in \calF'$ in an explicit decomposition into factors of the form $F_n^k$, where $n\in \nn$, $k \in \zz.$
\end{rem}
\begin{cor}
\label{cor-reduction-to-lamplighter}
    The group $G_{\calA}/\calF'$ is isomorphic to the (restricted) wreath product $\zz \wr \zz$ under the maps moving $s\calF'$ to the generator of the top (acting) group and $F\calF'$ to the generator of the bottom group.
\end{cor}
\begin{proof}
    Follows from Proposition \ref{prop-for-reduced-form-uniqueness} and the observation that modulo $\calF'$ the factors $F^{s^k}$ commute with each other.
\end{proof}
\begin{df} [$s$-part, $\calF$-part, $\calF'$-part]
    Let \begin{align*}
        g = s^x\left(  \prod_{i=1}^{t} (F^{s^{n_i}})^{k_i} \right)f,
    \end{align*}
    be given in reduced form. Then, we define
    \[
    g_s: = s^x, \text{ ~and call it the \textbf{$s$-part} of $g$;}
    \]
    \[
    g_{\calF}: = \prod_{i=1}^{t} (F^{s^{n_i}})^{k_i}, \text{~ and call it the \textbf{$\calF$-part} of $g$;}
    \]
   \[
    g_{\calF'}: = f, \text{~and call it the \textbf{$\calF'$-part} of $g$.}
    \]

\end{df}
\begin{thm}[An explicit presentation of $G_{\calA}$]
\label{thm_a-presentation-for-G_A}
For 
\[
\calA = \langle a_i; i \in \zz \mid [a_i, a_j]=1, i,j\in \zz; a_{2i}=1, a_i=a_{-i}, i \in \zz; \calR \rangle,
\]
where $\calR$ is any set of relators with respect to the alphabet $\{a_i^{\pm 1} \mid i \in \zz\}$, we have
\[
G_{\calA}= \langle F, s \mid \calR_1 \cup \calR_2 \cup \calR_3 \cup \calR_{\calA} \rangle,
\]
where
\[
\calR_1: = \{ [F_n, F^{s^m}] \mid m, n\in \zz \},
\]
\[
\calR_2: = \{ F_n^sF_n \mid n\in \zz \},
\]
\[
\calR_3: = \{ [F^{-1}, F^{s^n}]F_n \mid n\in \zz \},
\]
\[
\calR_{\calA}: = \{ F_{i_1}^{\varepsilon_1}F_{i_2}^{\varepsilon_2}\ldots F_{i_l}^{\varepsilon_l}\mid a_{i_1}^{\varepsilon_1}a_{i_2}^{\varepsilon_2}\ldots a_{i_l}^{\varepsilon_l} =_{\calA} 1, \varepsilon_1, \ldots, \varepsilon_l \in \{\pm 1\} \}.
\]
\end{thm}
\begin{proof}
    First of all, the fact that $\calR_1$, $\calR_2$, $\calR_3$, and $\calR_{\calA}$ are relators of the group $G_{\calA}$ follows from Propositions \ref{prop-key-identities} and \ref{prop-about-calF'-calA-isom}, respectively. Therefore, to finish the proof, it is enough to show that all identities in $G_{\calA}$ are consequences of those relators. To that end, we will observe that if a word $w \in \{F^{\pm 1}, s^{\pm 1}\}^*$ represents identity in $G_{\calA}$, then the reduced form of it contains only the $\calF'$-part in a form of a product of factors of the form $F_n^k$, $n\in \nn$, $k\in \zz$. See Remark \ref{rem-for-reduced-form}. Verification of whether or not the obtained $\calF'$-part is equal to the identity of $G_{\calA}$ is a consequence of the relators from $\calR_{\calA}$ because of Proposition \ref{prop-about-calF'-calA-isom}. Also, the validity of the swap operations, as pointed out in Remark \ref{rem-for-reduced-form}, is a consequence of the identities from Proposition \ref{prop-key-identities}. Finally, it is an easy observation that the identities from Proposition \ref{prop-key-identities} are consequences of the identities $\calR_1 \cup \calR_2 \cup \calR_3.$
\end{proof}
\section{The word problem and recursively enumerable presentations for $G_{\calA}$}
\label{sec-word_problem}
\begin{df}[Recursively presented and computable groups]
    Let $H$ be a group with countable set of enumerated generators $X_H=\{h_1, h_2, \ldots \}$. Then we say that $H$ is \textbf{computable} with respect to the enumerated set of generators $X_H$ if the word problem of $H$ is decidable with respect to $X_H$. In other words, it means that the set of finite collections of pairs of the form $(i_1, \varepsilon_1), \ldots, (i_l, \varepsilon_l)$ such that $h_{i_1}^{\varepsilon_1}\ldots h_{i_l}^{\varepsilon_l} = 1$ is a recursive set, where $i_j \in \nn$ and $\varepsilon_j \in \{\pm 1\}.$

    We also say that $H$ is \textbf{recursively presented} with respect to $X_H=\{h_1, h_2, \ldots \}$, if the set of finite collections of pairs $(i_1, \varepsilon_1), \ldots, (i_l, \varepsilon_l)$ such that $h_{i_1}^{\varepsilon_1}\ldots h_{i_l}^{\varepsilon_l} = 1$ is recursively enumerable (but not necessarily recursive.)
\end{df}
\begin{thm}[On the decidability of the word problem in $G_{\calA}$]
\label{thm-WP-in-G_A}
    The group $G_{\calA}$ has decidable word problem if and only if $\calA$ is computable with respect to the generating set $\{a_1, a_2, \ldots \}$.
\end{thm}
\begin{proof}
    For any element $g\in G_{\calA}$, one can algorithmically compute its reduced form $ g = s^x\left(  \prod_{i=1}^{t} (F^{s^{n_i}})^{k_i} \right)f$, where $f\in \calF'$ is explicitly given as a product of factors of the form $F_n$, $n\in \zz$. Since the parameters $x, n_1, \ldots, n_t, k_1, \ldots, k_t$ are uniquely defined by $g$ (see Proposition \ref{prop-on-reduced-form}), we get that $g=_{G_{\calA}}$ would imply that  $x= n_1=\ldots = n_t =  k_1= \ldots= k_t=0$, which is an algorithmically detectable condition regardless of the presentation of $\calA$. Therefore, the decidability of the word problem in $G_{\calA}$ reduced down to the decidability of the word problem for elements of $\calF'=\langle F_n, n \in \zz \rangle$ given as products of the factors $F_n, n \in \zz.$ Therefore, since by Proposition \ref{prop-about-calF'-calA-isom}, $F_n \leftrightarrow a_n$ establishes an isomorphism between $\calF'$ and $\calA$, we get that the decidability of the word problem in $G_{\calA}$ algorithmically reduced down to the computability of $\calA$ (with respect to its fixed presentation).
\end{proof}
\begin{thm}
\label{thm-RE-pres-for-G_A}
    The group $G_{\calA}$ has a recursively enumerable presentation if and only if the presentation of $\calA$ with respect to which $G_{\calA}$ is defined is recursively enumerated.
\end{thm}
\begin{proof}
    The `if' part is a direct consequence of Theorem \ref{thm_a-presentation-for-G_A}. The `only if' part is a direct consequence of Proposition \ref{prop-about-calF'-calA-isom}.
\end{proof}
\section{Some word-metric properties of $G_{\calA}$}
\label{sec-metric_properties}
We start with the following lemma.
\begin{lem}
\label{lem-length-from-reduced-form}
    Let $g = s^x\left(  \prod_{i=1}^{t} (F^{s^{n_i}})^{k_i} \right)f $ be given in its reduced form.  Then, $|g|_{\{F, s\}} \geq \max\{|k_1|, \ldots, |k_t| \}. $
\end{lem}
\begin{proof}
    Indeed, in obtaining a reduced form from a finite word in $\{F^{\pm 1}, s^{\pm 1} \}$, the number of $F^{\pm 1}$-s that appear in the {$\calF$-part} of the reduced word is staying the same as in the initial word or is decreasing due to free cancellations. Therefore, $|g| \geq \sum_{i=1}^t |k_i| \geq \max\{|k_1|, \ldots, |k_t| \}.$
    
\end{proof}
\begin{lem}
\label{lem-metric-2}
    Let $g = F_n^k$, $k\neq 0$. Assume that $a_n \notin \langle a_i; |i|<n \rangle$. Then, $$ |g|_{\{F, s\}}  \leq 8(n+\sqrt{|k|}).$$ 
    Moreover, if $a_n^k \in \calA$ is of length $|k|$ with respect to the generating set $\{a_i \mid i=1,2,\ldots \}$ of $\calA$, then, in addition, we have 
    $$n+\sqrt{|k|} \leq |g|_{\{F, s\}}.$$
    In particular, $|F_n|_{\{F, s\}} = 4n+4$.
\end{lem}
\begin{proof}
Without loss of generality, let us assume that $k>0$.
Let $l \in \nn$ be such that $l^2 \leq k < (l+1)^2$. Then, note that, by Property \ref{prop-key-identities}, $$F_n^k = F_n^{l^2} F_n^{k-l^{2}} = [F^l, (F^{s^n})^l][F^{k-l^{2}}, F^{s^n}].$$
Therefore, 
\begin{align*}
    |F_n^k|_S &\leq |[F^l, (F^{s^n})^l]|_S+|[F^{k-l^{2}}, F^{s^n}]|_S \\
&\leq (4l+4n) + 2(k-l^{2})+ 4n \leq 8n + 8 l \leq 8(n+\sqrt{k}).
\end{align*}
Now, assume that
 \begin{align}
	\label{eq-4444}
 F_n^k=(F^{s^{\alpha_1}})^{z_1} (F^{s^{\alpha_2}})^{z_2}  \ldots (F^{s^{\alpha_t}})^{z_t},
 \end{align}
after free reductions, is the shortest decomposition by $s^{\pm 1}$ and $F^{\pm 1}$ of $ F_n^k$, where $\alpha_i\neq \alpha_{i+1}$.  This, in particular, means that $|F_n^k| \geq \max\{|\alpha_i| \mid 1\leq i \leq t\} + \sum_{i=1}^t |z_i|. $

Then by the iterative application of {swap operations} of the form
\begin{align}
\label{eq-0404}
    (F^{s^{y}})^{z}(F^{s^{y'}})^{z'} = (F^{s^{y'}})^{z'}(F^{s^{y}})^{z} F_{|y-y'|}^{(-1)^yzz'}
\end{align}
on \eqref{eq-4444}, one can decompose $F_n^k$ as a product of elements of the form $F_{|\alpha_i-\alpha_j|}^{(-1)^{\alpha_i}z_iz_j}$, $1\leq i, j \leq t$. Now, recall that by Proposition \ref{prop-about-calF'-calA-isom}, for odd $i$, $F_i \leftrightarrow a_i$ establishes an isomorphism between $\calF$ and $\calA$. Therefore, since we assume that $a_n  \notin \langle a_i; |i|<n \rangle$ and that $a_n^k$ is of length $k$ with respect to the generators $\{a_i; i \in 2\zz+1 \}$ of $\calA$, we get that 
$\max\{|\alpha_i-\alpha_j| \mid 1\leq i, j \leq t\}\geq n$ and $\sum_{1\leq i, j \leq t} |(-1)^{\alpha_i}z_iz_j| \geq k$. Therefore, $\max\{|\alpha_i|+|\alpha_j| \mid 1\leq i,j\leq t\}\geq n$ and $\sum_{i=1}^t |z_i| \geq \sqrt{k}$, which implies that $F_n^k$ is of length at least $n+\sqrt{k}$, because the right-hand side of \eqref{eq-4444}, after free reductions, has length at least $t+|\alpha_1|+|\alpha_t|+\sum_{1\leq i<j \leq t}|\alpha_i-\alpha_j|.$ Finally, combination of those observations gives us $|F_n|_{\{F, s\}} = 4n+4$.
\end{proof}

Note that in proving the left-hand side inequality of Lemma \ref{lem-metric-2}, we in fact proved something more general, which we state below.
\begin{lem}
\label{lem-metric-new}
    Let $g \in \calF'$ such that $g \notin \langle F_i; |i|<n \rangle$. Then $n < |g|_{\{F, s\}}.$
\qed
\end{lem}
\section{Periodicity and residual finiteness of $G_{\calA}$}
\label{section-periodicity-and-RF}

Let $\calA$ be defined as in Section \ref{section-main-construction}. That is, let $$\mathcal{A} = \langle a_1, a_2, \ldots  \rangle$$ be an abelian group with generators $\{a_1, a_2, \ldots \}$. Let $$\mathcal{B} = \langle b_1, b_2, \ldots  \rangle$$ be an abelian group with a distinguished set of generators $\{b_1, b_2, \ldots \}$ such that $\calA \leq \calB$ and $\{b_1, b_2, \ldots \}$ \emph{forms a \textbf{basis}} for $\calA$, in a sense that for any $a\in \calA\setminus \{1\}$ there is a \emph{unique} collection $\{(i_1, k_1), (i_2, k_2), \ldots (i_l, k_l)\}$ such that $i_1< \ldots< i_l$  and $0<k_i<|b_i|$, $i=1, \ldots, l$, where $|b_i|$ is the order of $b_i$ in $\calB$, so that
$$a = b_{i_1}^{k_1} b_{i_2}^{k_2} \ldots b_{i_l}^{k_l}.$$  We denote $$\supp(a):= \{ b_{i_1}, \ldots, b_{i_l} \}$$
and
\[
|a|_{b_i} := \begin{cases}
    k_j, \text{~if $i=i_j$ for some $1\leq j \leq l$;}\\
    0, \text{~otherwise.}
\end{cases}
\]
Similarly, for $f \in \calF'$, assuming that $f$ corresponds to $a\in \calA$ under the map induced from $F_i \leftrightarrow a_i$ (see Proposition \ref{prop-about-calF'-calA-isom}), we denote 
\[
|f|_{b_i}:=|a|_{b_i}.
\]
More generally, for $g \in G_{\calA}$, we denote
$|g|_{b_i}: = |g_{\calF'}|_{b_i}.$

We will say that elements from $\{b_1, b_2, \ldots \}$ are \textbf{basis elements} with respect to $\calA$.
\begin{df}
    The generating element $b_n\in \{b_1, b_2, \ldots \}$ is called \textbf{periodic} of \textbf{valence} $l\in \nn$ with respect to $\calA$ if there exists a finite collection of pairs $(T_1, r_1)$, \ldots, $(T_l, r_l)$, called \textbf{the periods} of $b_n$, and integers $d_1, \ldots, d_l$, called \textbf{the degrees} of periods $(T_1, r_1)$, \ldots, $(T_l, r_l)$, respectively,  such that $T_i$ are positive \emph{even} integers and $0\leq r_i \leq T_i/2$ such that, for every $m \in \nn$, we have
    \[
    |a_m|_{b_n}= \sum d_{i}, \text{~where $1\leq i \leq l$ are those indices for which $m \equiv \pm r_i \mod T_i.$}
    \]
\end{df}
\begin{df}
    \label{def-periodic} 
    $\calA$ is called \textbf{periodic} with respect to $\calB$  if each basis element from $ \{b_1, b_2, \ldots \}$ is periodic with respect to $\calA \leq \calB$.
\end{df}


\begin{thm}
\label{thm-main-th-about-RF}
\label{thm-RF-follows-from-periodicity}
    If $\calA = \langle a_1, a_2, \ldots  \rangle$ is periodic with respect to the abelian group $\mathcal{B} = \langle b_1, b_2, \ldots  \rangle$, then $G_{\calA}$ is residually finite.
\end{thm}

We prove this theorem in the next two subsections (Subsections \ref{subsection-case-type-1} and \ref{subsection-case-type-2}). In those subsections, for a fixed $g\in G_{\calA} $, we separately consider the cases when $g\in G_{\calA} \setminus \calF'$ and when $g\in \calF'$, respectively.

\subsection{The $T$-normal form and the vector function $\vec\phi_{q,T}$}
\label{subsection-case-type-2} 
\begin{df}[The $T$-normal form] 
 \label{def-n-form}
 
 Let $g \in G_{\calA}$ be given in its reduced form as:
 \begin{equation}
\label{eq-n-form}
    g =s^x \left(  \prod_{i=1}^{t} (F^{s^{n_i}})^{k_i} \right)   f ,
\end{equation}
where $f\in \calF'$, $x, t \in \zz$, $t\geq 0$, the exponents $k_1, \ldots, k_{t}$  are non-zero integers, $n_1, \ldots, n_{t} \in \zz$  are pairwise different. Also, $t=0$ means $g = s^xf$. 

Let $T \in \mathbb{N}$ be a given positive integer. Then we say that the decomposition \eqref{eq-n-form} is a \textbf{$T$-normal form} for $g$ if the following property holds:
\begin{itemize}
    \item[] For  $1 \leq i_1 < i_2 \leq t$, we have
    $$ n_{i_1}\mod T \geq n_{i_2}\mod T$$ with respect to the ordering $0<1< \ldots < T-1$.
\end{itemize}
In particular, for $1 \leq i_1 < i_2 \leq t$, if $n_{i_1} \equiv n_{i_2} \mod T$, then for any $i$ such that $i_1 \leq i \leq i_2$, we have $n_{i_1} \equiv n_i \mod T$.
 \end{df}

\begin{lem}
\label{lem-uniqueness_of_N-form}
    For every even $T \in \nn$ and $g \in G_{\calA}$, $g$ can be represented in a $T$-normal form as in \eqref{eq-n-form}, and $x$ and $f$ are independent  of the choice of $T$ and the decomposition. Moreover, $(n_1, n_2, \ldots, n_t)$ and $(k_1, k_2, \ldots, k_t)$ are also unique up to transpositions of terms of the form $(F^{s^{m}})^{k}$ and $(F^{s^{n}})^{l}$ for $m \equiv n \mod T$.
\end{lem}
\begin{proof}
    Note that in $s^x \left(  \prod_{i=1}^{t} (F^{s^{n_i}})^{k_i} \right) f$, after rearranging the terms $(F^{s^{n_i}})^{k_i}$ by iterative application of the swap operations, one can obtain the reduced form of $g$. Therefore, it becomes evident that the assertion of the lemma is a direct consequence of the uniqueness of the reduced form, as in Proposition \ref{prop-for-reduced-form-uniqueness}.
\end{proof}
Let $p \in \mathbb{N}$ and $r \in \{ 0, 1, \ldots, T-1 \}$. Then, for $g \in G_{\calA}$ given in reduced form as 
 \[
 g = s^x\left(  \prod_{i=1}^{t} (F^{s^{n_i}})^{k_i} \right)f,
 \]
 define
$$ \phi_{q,T, r}(g) = \left( \sum_{\substack{1 \leq i \leq t \\ n_i \equiv r \mod T}} k_i \right) \mod q.$$ 

\begin{lem}
\label{lem-well-def-phi}
    $\phi_{q,T, r}(g)$ is well defined.
\end{lem}
\begin{proof}
    Follows from Proposition \ref{prop-for-reduced-form-uniqueness}.
\end{proof}
Define $\vec\phi_{q,T}: G_{\calA} \rightarrow \mathbb{Z}_q^T$  as
$$\vec\phi_{q,T} (g) = \left(\phi_{q,T, 0}(g), \phi_{q,T, 1}(g), \ldots, \phi_{q,T, T-1}(g) \right) \in \mathbb{Z}_q^{T},$$
where by $\mathbb{Z}_q$ we denote the additive group of order $q$ with elements $\{0, 1, \ldots, q-1\}$.
By Lemma \ref{lem-well-def-phi}, $\vec\phi_{q,T}$ is well-defined. 

Let us denote 
$N_{q,T}:=\{ g \in G_{\calA} \mid \vec \phi_{q,T}=\vec 0 \} \subset G_{\calA}.$ It is straightforward to see that $N_{q,T}$ is a normal subgroup of $G_{\calA}.$ Moreover, it is also straightforward to check that
\[
(G_{\calA}/\calF') \big/ (N_{q,T}/(N_{q,T}\cap \calF')) \simeq \zz_q \wr \zz_T.
\]

The proofs of  Lemmas \ref{lem-additivity-of-BARphi}-\ref{lem-additivity-|*|_r} are straightforward.
\begin{lem}
    \label{lem-additivity-|*|_s}
    Let $g_1, g_2 \in G_{\calA}$. Then, $|g_1^{-1}|_s = -|g_1|_s $ and 
    \[
    |g_1g_2|_s = |g_1|_s + |g_2|_s.
    \]
\end{lem}
\begin{lem}
\label{lem-additivity-of-BARphi}
    If $g_1, g_2 \in G_{\calA}$ such that $|g_2|_{s}\equiv 0 \mod T$, then 
    \[
    \vec\phi_{q,T}(g_1g_2) =\vec\phi_{q,T}(g_1) + \vec\phi_{q,T}(g_2).
    \]
\end{lem}
\begin{lem}
\label{lem-on-|*|_N_r}
    If $g_1, g_2 \in G_{\calA}$ such that $|g_2|_{s}\equiv 0 \mod T$ and  $\vec\phi_{q,T}(g_2) = \vec 0$. Then 
    \[
    |g_1g_2|_{b_r} \equiv |g_1|_{b_r}  \mod q.
    \]
\end{lem}
\begin{lem}
    \label{lem-additivity-|*|_r}
     Let $g_1, g_2 \in G_{\calA}$ be such that $\vec\phi_{q,T}(g_1) = \vec\phi_{q,T}(g_2)$. Then,
     \[
     \vec\phi(g_1g_2^{-1}) = \vec 0 \text{~and~} |g_1g_2^{-1}|_{b_r} \equiv |g_1|_{b_r}  - |g_2|_{b_r} \mod q.
     \]
\end{lem}

Let $r \in \nn$ be such that $b_r$ is a basis element with periods $(T_1, r_1), \ldots, (T_k, r_k)$ having degrees of periods $d_1, \ldots, d_k$, respectively. Let $T \in \nn$ be a positive \emph{even} integer such that for all $1\leq i \leq k$, $T_i \mid T$. Let $q>1$ be a fixed integer such that it is a multiple of the order of $b_r$ in $\calB$ in the case the latter is finite. 

Throughout this section, we will always assume that $q, T$ satisfy these conditions with respect to a fixed $b_r$, and $b_r$ will always be a basis element with respect to $\calA\leq \calB$.

Define $N_{q,T, b_r}$ to be the collection of elements $g$ from $G_{\calA}$ such that $\vec{\phi}_{q,T}(g)$ is the zero vector of $\mathbb{Z}_q^T$,  $|g|_s$ is a multiple of $T$, and $|g|_{b_r}$ is a multiple of $q$. In other words, we have the following notation:
\begin{df}
\label{def-N_q,n,r}
    $$N_{q,T, b_r} := \{ g \in G_{\calA} \mid T \text{~divides~}|g|_s {~\&~} q \text{~divides~} |g|_{b_r} {~\&~} \vec{\phi}_{q,T}(g) = \vec 0_{\mathbb{Z}_q^T} \} . $$
\end{df}

\begin{lem}
\label{lem-normality-N_q,n,r}
    $N_{q,T, b_r}$ is a normal subgroup of $G_{\calA}$. 
\end{lem}
 \begin{proof}
 ~\\
 \noindent
    \underline{Closure under multiplication:} 
     Follows directly from Lemmas \ref{lem-additivity-|*|_s}, \ref{lem-additivity-of-BARphi}, and \ref{lem-additivity-|*|_r}.

     \noindent
     \underline{Closure under inverse:}
      Follows directly from Lemmas \ref{lem-additivity-|*|_s}, \ref{lem-additivity-of-BARphi}, and \ref{lem-additivity-|*|_r}.\\
     \noindent
     \underline{Closure under conjugation:} Follows from Lemmas \ref{lem-additivity-|*|_s}, \ref{lem-additivity-of-BARphi}, \ref{lem-on-|*|_N_r}, and \ref{lem-additivity-|*|_r}. Indeed, if $g  \in N_{q,T, b_r}$ and $h \in G_{\calA}$, then by Lemma \ref{lem-additivity-|*|_s} we get $|hgh^{-1}|_{s} = 0$. From Lemma \ref{lem-additivity-of-BARphi}, we get $\vec\phi_{q,T}(hg) =\vec\phi_{q,T}(g) $. Therefore, from Lemma \ref{lem-additivity-|*|_r} we get $\vec\phi_{q,T}(hgh^{-1}) = \vec 0$ and $|hgh^{-1}|_{b_r} \equiv |hg|_{b_r} - |h|_{b_r} \mod q$. Finally, by Lemma \ref{lem-on-|*|_N_r}, we have $|hg|_{b_r} \equiv |h|_{b_r} \mod q.$
     \end{proof}
\begin{props}
\label{props-characterization-of-N_q,n,r}
   Let $g_1, g_2 \in G_{\calA}$. Then, $g_1 \equiv g_2 \mod N_{q,T, b_r}$ if and only if $|g_1|_{s}\equiv |g_2|_{s}  \mod T$, $|g_1|_{b_r} \equiv |g_2|_{b_r} \mod q$ and $\vec\phi_{q,T}(g_1) = \vec\phi_{q,T}(g_2)$.
\end{props}
\begin{proof}
    First, assume that $g_1 \equiv g_2 \mod N_{q,T, b_r}$. Then, by the definition of $N_{q,T, b_r}$, we get $\vec\phi_{q,T}(g_1^{-1}g_2) = \vec 0.$ Also,  combining with Lemma \ref{lem-additivity-|*|_s}, we get $|g_1|_{s}\equiv |g_2|_{s}  \mod T$ and, in particular, $|g_1^{-1}g_2|_s=0$. Therefore, applying Lemma \ref{lem-additivity-of-BARphi}, we get 
    \[
    \vec\phi_{q,T}(g_2) = \vec\phi_{q,T}(g_1 (g_1^{-1}g_2)) = \vec\phi_{q,T}(g_1) + \vec\phi_{q,T}(g_1^{-1}g_2),
    \]
    which implies $\vec\phi_{q,T}(g_2) = \vec\phi_{q,T}(g_1)$. Finally, applying Lemma \ref{lem-additivity-|*|_r}, we get $0 \equiv  |g_1^{-1}g_2|_{b_r} \equiv |g_1|_{b_r} - |g_2|_{b_r} \mod q, $
    which implies $|g_1|_{b_r} \equiv |g_2|_{b_r} \mod q.$

    Now, assume that $|g_1|_{s}\equiv |g_2|_{s}  \mod T$, $|g_1|_{b_r} \equiv |g_2|_{b_r} \mod q$ and $\vec\phi_{q,T}(g_1) = \vec\phi_{q,T}(g_2)$. Then, by Lemma \ref{lem-additivity-|*|_s}, we get $|g_1^{-1}g_2|_s = 0.$ Therefore, applying Lemma \ref{lem-additivity-of-BARphi}, we get 
    \[
    \vec\phi_{q,T}(g_2) = \vec\phi_{q,T}(g_1 (g_1^{-1}g_2)) = \vec\phi_{q,T}(g_1) + \vec\phi_{q,T}(g_1^{-1}g_2),
    \]
    which implies $\vec\phi_{q,T}(g_1^{-1}g_2)=\vec 0.$ And, applying Lemma \ref{lem-additivity-|*|_r}, we get $  |g_1^{-1}g_2|_{b_r} \equiv |g_1|_{b_r} - |g_2|_{b_r} \equiv 0 \mod q. $ Thus we conclude that $g_1 \equiv g_2 \mod N_{q,T, b_r}.$
\end{proof}
The next proposition is an immediate corollary from Proposition \ref{props-characterization-of-N_q,n,r}.
 \begin{props}
\label{props-index-of-N_q,n,r}
    Let $b_r$, $T$ and $q$ be defined as above. Then $[G_{\calA} : N_{q,T, b_r}] \leq T q^{T+1}.$ Moreover, in the special case when $b_r$ is of infinite order in $\calB$ or else $q$ coincides with the order of $b_r$, the inequality turns into equality.
     \end{props}

From Proposition \ref{props-characterization-of-N_q,n,r} we immediately get the following.

\begin{cor}
    \label{cor-RF_of_G_AND_upper_estimate_for_individual_RFG}
    If $\calA$ is periodic with respect to $\calB$ and additionally $\calB$ forms a basis for $\calA$, then $G_{\calA}$ is residually finite.
\end{cor}
\begin{proof}
    Let $g \in G\setminus \{1\}$. We want to show that there is a normal subgroup $N_g \triangleleft G_{\calA}$ of finite index such that $g \notin N_g$.

    Due to Proposition \ref{prop-RFG-TypeI_and_II_cases}, it is enough to consider the case when $g \in \calF'$. For this case, for some basis element $b_r$ of $\calB$, $d:=|g|_{b_r}\neq 0$.  Let $q\in \nn$ be the smallest positive integer such that $q\not| d$ if $b_r$ is of infinite order and otherwise $q$ is the smallest divisor of the order of $b_r$ such that $q\not| d$. Then, $g \notin N_{q,T, b_r}$. Therefore, we can take $N_g=N_{q,T, b_r}$
\end{proof}

\subsection{The residual finiteness depth for $g\in G_{\calA} \setminus \calF'$.}
\label{subsection-case-type-1}

Let $\langle a \rangle$ and $\langle b \rangle$ be infinite cyclic groups.
Let $p$ be a prime number. Let us consider the lamplighter group $L_p: = (\langle a \rangle / \langle a^p \rangle) \wr \langle b \rangle \simeq \zz/p \wr \zz.$

\begin{thm}[See \cite{bou-rabee--lampligher}]
    \label{lem-by-Bour-Rabee-etal}
    $\rf_{L_p, \{a, b\}}(n) \leq 4p^2 (n+1)^2. $
\end{thm}
\begin{proof}
The proof is given in \cite{bou-rabee--lampligher}. See Theorem 1 and Subsection 3.1 for its proof in \cite{bou-rabee--lampligher}. The only remark to add here is that the length of elements in $L_p$ with respect to the generators $\{a, b\}$ coincides with the length with respect to $X$ considered in their proof. In fact, we have $\rf_g \leq \log(|x|)$.
\end{proof}

\begin{props}
    \label{prop-RFG-TypeI_and_II_cases}
    If $g\in G_{\calA}\setminus \{\calF'\}$, then $\rf_g \leq 16 |g|_{\{F, s\}}^4$, where $\rf_g$ is with respect to the standard generating set $\{F, s\}$.
\end{props}
\begin{proof}
    Let $g \in \calF \setminus \calF' \in G_{\calA}$ be an element given in its reduced form 
 $g = s^x \left(  \prod_{i=1}^{t} (F^{s^{n_i}})^{k_i} \right)f$ as in \eqref{eq-reduced-form}.
 
 First, assume that $x \neq 0$. Then, the claim follows from the observation that $s \mapsto s$, $F \mapsto 1$ induces an epimorphism from $G_{\calA}$ to the infinite cyclic group $\langle s \rangle$, which implies the inequality of the residual finiteness depth for $g$.\\

Now, let $g \in \calF \setminus \calF' \in G_{\calA}$ be an element given in its reduced form 
 $g = \left(  \prod_{i=1}^{t} (F^{s^{n_i}})^{k_i} \right)f$, where $t>0$.
 
 Consider the restricted wreath product $$L:= \langle a \rangle \wr \langle b \rangle \simeq \zz \wr \zz.$$
Note that  the map $\phi: F \mapsto a, s \mapsto b$ induces an epimorphism between $G_{\calA}$ and $L$ with the kernel $\ker(\phi) = \calF'$ (see Corollary \ref{cor-reduction-to-lamplighter}).

 Let us denote $k: = \max\{|k_1|, |k_2|, \ldots, |k_t|\}.$ Let $p$ be a prime such that \begin{align}
     \label{eq-for-p-in-TypeII}
k<p\leq 2k.
\end{align}
Denote by $\psi_p: L \rightarrow L_p$ the natural epimorphism between $L$ and $L_p$. Note that, due to \eqref{eq-for-p-in-TypeII}, $\psi_p(\phi(g))$ is non-trivial. Therefore, due to Lemma \ref{lem-by-Bour-Rabee-etal}, we obtain $\rf_g \leq 4p^2(|g|_{F, s}+1)^2$.
Therefore, due to the observation that $|g|_{F, s} \geq k > p/2$ (see Lemma \ref{lem-length-from-reduced-form}) and \eqref{eq-for-p-in-TypeII}, we get
$$\pushQED{\qed} \rf_g \leq 16(|g|_{F, s})^4.\qedhere \popQED$$ 

\end{proof}

\section{A more general setting for periodicity via normal forms}
\label{subsec-more-general-periodicity}
In this subsection, we generalize the concept of periodicity of $\calA$ with respect to $\calB$ for the case when the generating elements from the presentation of $\calB = \langle b_1, b_2, \ldots \rangle$ do not necessarily form a basis in the sense described in Section \ref{section-periodicity-and-RF}. Extending the ideas from  Section \ref{section-periodicity-and-RF}, we will show that under some conditions on the presentation of $\calB$, $G_{\calA}$ is residually finite. To that end, we will consider a normal form for $\calB$ that satisfies additional conditions.

Let $\calB = \langle b_1, b_2, \ldots \mid [b_i, b_j]=1, i, j\in \nn; \calR \rangle$, where $\calR$ is a set of relator words from $\{b_1^{\pm 1}, b_2^{\pm 1}, \ldots \}^* $ satisfying the condition that for each $i \in \nn$, $b_i^{\pm 1}$ appears only in finitely many words from $\calR$.

Let us consider a \emph{normal form} $\nu: \calB \rightarrow \{b_1^{\pm 1}, b_2^{\pm 1}, \ldots \}^*$, that is, $\nu$ is such that for every $b\in \calB$, $\nu(b)$ is a word from $\{b_1^{\pm 1}, b_2^{\pm 1}, \ldots \}^*$ that represents $b$ in $\calB$. Furthermore, we assume that for $b\in \calB$, $\nu(b)$ is of the form
\[
\nu(b) = b_{n_1}^{l_1} \ldots b_{n_k}^{l_k}, ~n_1 < \ldots < n_k, ~ b_{n_i}^{l_i}\neq 1, i=1, 2, \ldots, k.
\]
For $n \in \nn$, we denote 

 \[
    |b|_{n, \nu} = \begin{cases}
                       l_i    & \mbox{if $n=n_i$ for some $1\leq i \leq k$,}\\
                       0  & \mbox{otherwise.}         
                       \end{cases}
           \]

Also, let us denote 
\[
\supp_{\nu}(b): = \{b_{n_1}, \ldots, b_{n_k} \}.
\]

Since, by Proposition \ref{prop-about-calF'-calA-isom}, $\calF' \simeq \calA \leq \calB$ under $F_i \leftrightarrow a_i$, the above definition extends to the definition of $|f|_{n, \nu}$ for $f \in \calF'$. Similarly, we can define $\supp_{\nu}(f): = \supp_{\nu}(a).$
\begin{df}[Condition (*)]
Let $\calB = \langle b_1, b_2, \ldots \mid [b_i, b_j]=1, i, j\in \nn; \calR \rangle$, where $\calR$ is a set of relator words from $\{b_1^{\pm 1}, b_2^{\pm 1}, \ldots \}^* $. Let $\nu: \calB \rightarrow \{b_1^{\pm 1}, b_2^{\pm 1}, \ldots \}^*$ be a fixed normal form of $\calB$. We say that the presentation of $\calB$ satisfies \textbf{Condition (*)} with respect to $\nu$ if the presentation of $\calB$ simultaneously satisfies the two conditions described below.
\begin{enumerate}[label=(\roman*)]
    \item  For each $i_0 \in \nn$, there exists a finite collection of indices $i_1, i_2, \ldots, i_k$ such that $\calB$ can be written as a direct product of the form
    \[\calB = \langle b_{i_0}, b_{i_1}, b_{i_2}, \ldots, b_{i_k} \rangle \times \calB'\] \text{~ for $\calB' := \langle b_j; j \notin \{i_0, i_1, \ldots, i_k \} \rangle \leq \calB$};
    \item For every $b\in \calB$, $n \in \nn$ and $D \in \nn$, there exists $d>D$ such that if the order of $b$ is finite, then $d$ is a multiple of it, and such that for every $b' \in \calB$, we have
    \[
|b (b')^d|_{b_{n}, \nu} \equiv |b|_{b_{n}, \nu} \mod d.
\]
\end{enumerate}
\end{df}
\begin{rem}
    Note that in the part (ii) of the above definition, if $n = i_0$, where $i_0$ is as in the part (i), then without loss of generality we could assume that $b' \in \langle b_{i_0}, b_{i_1}, b_{i_2}, \ldots, b_{i_k} \rangle.$ Also, note that if all elements of $\calB$ are of finite order, then part (ii) follows from part (i), as one simply can take $d$ to be a multiple of the product of orders of $b_{i_0}, b_{i_1}, b_{i_2}, \ldots, b_{i_k}$.
\end{rem}
\begin{df}[A more general setting for periodicity]
\label{def-condition-(*)}
Let $\calA \leq \calB = \langle b_1, b_2, \ldots \mid [b_i, b_j]=1, i, j\in \nn; \calR \rangle$. Then $\calA$ is periodic with respect to $\calB$ if for each $b_i$, $i\in \nn$, there exists a pair $(T_i, r_i)$, called \emph{the period} of $b_i$, satisfying $T_i \in 2 \nn$ and $0 \leq r_i \leq T_i/2$, and an integer $d_i \in \zz$ such that for every $n \in \zz$,
\begin{align*}
    a_n = \prod_{i=1}^{\infty} b_i^{d_i \varepsilon_i}, 
\end{align*}
where
\[
\varepsilon_i = \begin{cases}
                       1   & \mbox{if $n\equiv \pm r_i \mod T_i$,}\\
                       0  & \mbox{otherwise.}         
                       \end{cases}
\]
\end{df}

\begin{thm}[Residual finiteness of $G_{\calA}$, when $\calA$ is periodic w.r.t. $\calB$]
\label{thm-a_general_case_for_G_A_being_RF}
Let us assume that $\calA \leq \calB = \langle b_1, b_2, \ldots \mid [b_i, b_j]=1, i, j\in \nn; \calR \rangle$ is such that $\calA$ is periodic with respect to $\calB$ and the presentation of $\calB$ satisfies Condition (*) defined above with respect to some normal form $\nu: \calB \rightarrow \{b_1^{\pm 1}, b_2^{\pm 1}, \ldots \}^*$. Then, $G_{\calA}$ is residually finite.
\end{thm}
\begin{proof}
    Let $g \in G_{\calA} \setminus \{1\}$. We want to show that there exists $N_g \triangleleft G_{\calA}$ such that $g \notin N_{g}$ and $[G_{\calA} : N_g] \leq \infty$. Given the discussion from Section \ref{section-periodicity-and-RF}, without loss of generality, we can assume that $g \in \calF'$. Then under $F_i \leftrightarrow a_i$ (see Proposition \ref{prop-about-calF'-calA-isom}), $g$ corresponds to $a \in \calA \setminus \{1\} \subseteq \calB \setminus \{1\}.$

     Let
    \[
    \nu(a) = b_{n_1}^{l_1} \ldots b_{n_k}^{l_k}, ~n_1 < \ldots < n_k, ~l_i \neq 0, i=1, 2, \ldots, k.
    \]
    Let $i_1, \ldots, i_k$ be such that
     \[
     \calB = \langle b_{n_1}, b_{i_1}, b_{i_2}, \ldots b_{i_k} \rangle \times \calB'\] \text{~ for $\calB':=\langle b_j; j \notin \{n_1, i_1, \ldots, i_k \} \rangle \leq \calB$.} The existence of such a decomposition of $\calB$ follows from the part (i) of Condition (*).
     
    Let $d \in \nn$ be such that $d > |l_1|$ and satisfies the part (ii) of Condition (*) if we take there $n=n_1$.

    Let $(T_{n_1}, r_{n_1})$, $(T_{i_1}, r_{i_1}), \ldots, (T_{i_k}, r_{i_k})$ be the periods of $b_{n_1}$, $b_{i_1}, \ldots, b_{i_k}$, respectively. Let $T=\lcm(T_{n_1}, T_{i_1}, \ldots, T_{i_k}) $. Let $d_{n_1}, d_{i_1}, \ldots, d_{i_k}$ satisfy part (ii) of Condition (*) for $b_{n_1}, b_{i_1}, b_{i_2}, \ldots b_{i_k} $, respectively. Denote $d = \lcm(d_{n_1}, d_{i_1}, \ldots, d_{i_k})$. Let us define 
    \begin{align*}
           N_g: = \{ h \in G_{\calA} \mid T& \text{~divides~}|h|_s \\
    &d_{n_1} \text{~divides~} |h|_{b_{n_1}, \nu}, \text{~and~} d_{i_j} \text{~divides~} |h|_{b_{i_j}, \nu}, j=1, \ldots, k\\ 
    & \vec{\phi}_{d,T}(h) = \vec 0_{\mathbb{Z}_d^{T}} \}. 
    \end{align*}
    Note that the definition of $N_g$ is in  analogy with the definition of $N_{q,T, b_r} $ from Definition \ref{def-N_q,n,r}. The reason for putting divisibility conditions on all $|h|_{b_{n_1}, \nu}, |h|_{b_{i_1}, \nu}, \ldots, |h|_{b_{i_k}, \nu}$ simultaneously is because if so, then due to the part (ii) of Condition (*), this property stays closed under taking product and inverses, hence making $N_g$ a group. The proof that $N_g$ is a normal subgroup of $G_{\calA}$ of finite index such that $g \notin N_g$ is also in complete analogy with the proof of the analogous statement for $N_{q,T, b_r}$ (in particular, see Lemma \ref{lem-normality-N_q,n,r}).
    
\end{proof}

\section{Robustness of the residual finiteness growth functions: proof of Theorem \ref{thm-RFG-spectra}}
\label{section-robustness-torsion_case}
\label{section-RFG-robustness}

The goal of this section is to prove the following theorem.
\begin{thm}[Theorem \ref{thm-RFG-spectra}]
\label{thm-RFG-spectra-COPY}
    For any non-decreasing $f:\nn \rightarrow \nn$ such that for some $\varepsilon>0$, $f(n) \geq \exp{\varepsilon n \log{n}}$ there exists a residually finite group $G_f$ such that $\rf_{G_f}(n) \simeq f(n).$ Moreover, $G_f$ can be found among  two-generated solvable groups of derived length $3$.
\end{thm}
 Throughout this section, we define
\[
\calP = \{ p_1 < p_2< \ldots \}
\]
\[
\calQ = \{ q_1 < q_2< \ldots \}
\]
to be increasing sets of {odd} prime numbers, which will be described in more details in Subsection \ref{subsctn-robust-RFGFs}. 
We also define
 $$\calB = \langle b_i \mid [b_i, b_j]=1, b_{i}^{q_i} = 1 \rangle$$
and $\calA = \langle a_i, i \in \zz \rangle < \calB$ such that  $b_m \in supp(a_n)$ \text{~ if and only if~} $n\equiv p_m \mod 2p_m.$ More precisely, we define
\begin{align*}
 \label{eq-calA_RFG_Spectra}
     \calA= \langle a_i, i\in \zz \mid &[a_i, a_j]=1, i, j \in \nn; \nonumber\\
      &a_{2n}=1;\nonumber\\
     & a_i=a_{-i}, i \in \zz; \nonumber \\
    &a_{p_{i}}=b_{i}, i \in \nn;\\
      &a_{2n+1}=\prod_{\substack{p\mid 2n+1\nonumber\\
     p \in \calP}} a_p, n\in \zz \rangle < \calB. \nonumber
 \end{align*}
Note that $\calA $ is periodic with respect to $\calB$, where the periods are the elements of $\calP$ and for $m\in \nn$,  the period of $b_m$ is $(2p_m, p_m)$ for $p_m \in \calP$. Therefore, by Theorem \ref{thm-RF-follows-from-periodicity}, the group 
$
G_{\calA},
$
defined as in Sections \ref{section-main-construction}, \ref{section-presentation_of_G_A}, is residually finite.

\subsection{Computing $\rf_g$ for $g=F_{p}$}
Throughout this subsection, we define $g\in G_{\calA}$ as
\[
g: = F_{p},
\]
for some $p:=p_i \in \calP$.
We also denote the period of $a_{p_i}$ by $T$, that is
\[
T:=2p_i
\]
We also denote
\[
r:=p_i
\]
and we denote the order of $g$ (i.e., the order of $a_{p_i}$) by $q$, that is
\[
q:=q_i.
\]
 \begin{lem}
 \label{lem-aux-6.1}
     Let $s^kh \in N_g$ for $h \in \calF$. Then there exists $h' \in \calF$ such that $\vec\phi_{q,T}(h')=\vec\phi_{q,T}(h)$ and $l \in \zz$ such that $T$ divides $l$ and $s^lh' \in N_g$.
 \end{lem}
 \begin{proof}
     $s^kh \in N_g$ implies $s^kh^{s^{k}} \in N_g$, which implies that $( s^k h^{s^{k}})(s^kh) \in N_g$. This means $s^{2k}h^2  \in N_g$, which implies $s^{2k} (h^2)^{s^{k}}  \in N_g$, and the latter implies $(s^{2k} (h^2)^{s^{k}} )(s^kh) = s^{3k} h^3  \in N_g$. Continuing this way, we get $s^{kx} h^x  \in N_g$ for all $x \in \nn$. Now choose $x$ such that $x \equiv 0 \mod T$ and $x \equiv 1 \mod q$ (such $x$ exists, because $\gcd(q,T)=1$). Then take $l:=kx$, $h':=h^x$ . Then, applying the {lemma on additivity of $\phi$} (Lemma \ref{lem-additivity-of-BARphi}), we get $\vec\phi_{q,T}(h') = x\vec\phi_{q,T}(h) = \vec\phi_{q,T}(h)$ and $s^lh' \in N_g$, as desired.  
 \end{proof}

\begin{cor}
    The set $ \{ \vec\phi_{q,T}(h) \mid h \in N_g \} \subseteq \zz_q^T$ is a linear subspace of $\zz_q^T.$
\end{cor}
\begin{proof}
    Follows immediately from Lemmas \ref{lem-aux-6.1} and \ref{lem-additivity-of-BARphi}.
\end{proof}
\begin{lem}
    \label{lem-r-proj-is-0}
    Let $f\in N_g \cap \calF'$. Then $|h|_{b_i}=0.$
\end{lem}
\begin{proof}
    Indeed, let $f = F_{p_i}^k \prod_{j=1}^t F_{p_{n_j}} \in N_g$, where $n_1, \ldots, n_t \neq i$. By the method of contradiction, assume that $k \not\equiv 0 \mod q_i$. Let $Q=\prod_{j=1}^t q_{n_j}$ and let $x \in \nn$ is such that $Qkx \equiv 1 \mod q_i$ (such $x$ exists since $q_i$ is prime and $\gcd(Q, q_i)=1$.) Then, $f^{Qkx}=F_{p_i} \in N_g$. A contradiction.
\end{proof}

\begin{lem}
\label{lem-bar-Ng-is-0}
    $\vec\phi_{q,T} (N_g) = \{\vec 0\}$. 
\end{lem}
\begin{proof}
    By the method of contradiction, assume that $\vec\phi_{q,T} (N_g) \neq \{\vec 0\}$. Then, there exists $h\in N_g$ such that $\vec \phi_{q,T}(g) \neq \vec 0$. In particular, this means that for some $0\leq r < T$, $ \phi_{q,T, b_r}(g) \neq 0$. By Lemma \ref{lem-aux-6.1}, without loss of generality we can assume that $T \mid k$. Now, let us consider $h': = [ F^{s^{p_i+r}}, h]$. Since $h\in N_g \trianglelefteq G_{\calA}$, we get $h' \in N_g$. However, note that $|h'|_{b_i}=  \phi_{q,T, b_r}(g) \neq 0$, which leads to contradiction due to Lemma \ref{lem-r-proj-is-0}. Thus we get $\vec\phi_{q,T} (N_g) = \{\vec 0\}$. 
    \end{proof}

\begin{lem}
\label{lem-s-part-Ng}
    If $h \in N_g$, then $|h|_s$ is a multiple of $T$.
\end{lem}
\begin{proof}
    Indeed, by the method of contradiction, assume that $h=s^xh_1 \in N_g$, where $h_1\in \calF$ and $x\in \zz$ such that for some $0<r\leq T-1$, $x \equiv r \mod T$. 

    The normality of $N_g$ implies that $F (s^xh_1) F^{-1} = s^x F^{s^{-x}}h_1 F^{-1}  \in N_g.$ But 
    \[
    \vec\phi_{q,T} (s^x F^{s^{-x}}h_1 F^{-1}) = \vec\phi_{q,T} (F^{s^{-x}}h_1 F^{-1} ) \neq \vec 0,
    \] as (for $x \equiv r \not\equiv 0 \mod T$) its $(T-1)$-th coordinate is equal to $q-1$ (and the $(T-r)$-th coordinate is equal to $1$). Thus, because of Lemma \ref{lem-bar-Ng-is-0}, we arrive to a contradiction.
\end{proof}
\begin{lem}
\label{lem-step-towards-index-of-Ng}
    Let $g_1, g_2 \in G_{\calA}$ be such that either $|g_1|_{s} \not\equiv |g_2|_{s} \mod T$ or $\vec\phi_{q,T} (g_1) \neq \vec\phi_{q,T} (g_2)$. Then, $g_1 \not\equiv g_2 \mod N_g$.
\end{lem}
\begin{proof}
    If $|g_1|_{s} \not\equiv |g_2|_{s} \mod T$, then it follows from Lemmas \ref{lem-additivity-|*|_s} and Lemma \ref{lem-s-part-Ng} that  $g_1 \not\equiv g_2 \mod N_g$. Now, assume that $|g_1|_{s} \equiv |g_2|_{s} \mod T$. Then, by Lemma \ref{lem-additivity-|*|_s}, $|g_1^{-1}g_2|_{s} \equiv 0 \mod T$, hence applying Lemma \ref{lem-additivity-of-BARphi}, we get $\vec\phi_{q,T}(g_2) = \vec\phi_{q,T}(g_1 (g_1^{-1}g_2)) = \vec\phi_{q,T}(g_1) + \vec\phi_{q,T}(g_1^{-1}g_2)$, which implies that $\vec\phi_{q,T}(g_1^{-1}g_2) =\phi_{q,T}(g_2) - \phi_{q,T}(g_1) \neq \vec 0$, which combined with Lemma \ref{lem-bar-Ng-is-0} implies that $g_1 \not\equiv g_2 \mod N_g.$
\end{proof}
\begin{cor}
\label{cor-on-index-of-Ng}
    $[G_{\calA}:N_g]\geq Tq^{T+1}.$ 
\end{cor}
\begin{proof}
By Lemma \ref{lem-step-towards-index-of-Ng}, we immediately get that $[G_{\calA}:N_g]\geq Tq^{T}$, as elements $h \in G_{\calA}$ for which $(|h|_{s}, \vec\phi_{q,T} (h) ) $ are pair-wise different must belong to different cosets of $T_g.$ To complete the proof, we will show that for any $h \in G_{\calA}$, there exist at least $p$ different elements $h' \in G_{\calA}$ pairwise belonging to different cosets of $N_g$, such that $|h|_{s} \equiv |h'|_{s} \mod T$ and $ \vec\phi_{q,T} (h)  =  \vec\phi_{q,T} (h') $. Indeed, consider the elements of the form $h' = g F_N^r$, $0 \leq r <q$. Then, $h'$ defined this way satisfied $|h|_{s} \equiv |h'|_{s} \mod T$  and $ \vec\phi_{q,T} (h)  =  \vec\phi_{q,T} (h') $, while pairwise belonging to different cosets of $N_g.$ Thus we conclude $[G_{\calA}:N_g]\geq Tq^{T+1}.$ 

\end{proof}
Combining Corollary \ref{cor-on-index-of-Ng} with Proposition \ref{props-index-of-N_q,n,r}, we get the following.
\begin{props}
    \label{props-RFG-for-Fp}
    $\rf_g= Tq^{T+1}.$
\end{props}

\subsection{A construction of a robust family of residually finite groups for proving Theorem \ref{thm-RFG-spectra}: further specification of the group $G_{\calA}$}
\label{subsctn-robust-RFGFs}

Let $f: \nn \rightarrow \nn$ be a non-decreasing function such that $$f(n)> n^{\lambda(2n+1)+1},$$
where $0<\lambda<1$ is a given constant.

Let us denote $$\eta(n) = \sqrt[2n+1]{\frac{f(n)}{n}}. $$

{Note that
\begin{align}
    \label{eq-aux-6.2}
    \eta(n)>n^{\lambda}.
\end{align}
}
Let $3=p_1<p_2<p_3< \ldots$ be an increasing sequence of prime numbers such that 

\begin{align}
    \label{eq-p-s}
16^{\lambda^{-1}}<\frac{p_{i+1}}{p_i} < 32^{\lambda^{-1}}.
\end{align}
In particular, $p_i> (16^{\lambda^{-1}})^i$. (Existence of such a sequence of primes follows from standard generalizations of Bertrand's Postulate.)

\begin{lem}
    \label{lem-on-existence-of-q_i-s}
    \label{lem-on-q-s}
    There exist pairwise different prime numbers $q_1, q_2, \ldots$ such that for some $0<\mu<1$ and for all $m \in \nn$, we have $\eta(p_m)\leq q_m \leq \eta(p_m)+\eta(p_m)^{\mu}. $
\end{lem}
\begin{proof}
    By \cite{bakerETal_2001}, for each $x\in \nn$, the interval $[x, x+x^{0.6}]$ contains at least one prime number. Given this, an easy induction shows that for $k\in \nn$, the interval $[x, x+3^k x^{0.6}]$ contains at least $k$ different primes. 

    Now, for each $m\in \nn$, the interval $[\eta(p_m), \eta(p_m)+3^{m} \eta(p_m)^{0.6}]$ contains at least $m$ different primes. On the other hand, since $\eta(n)>n^{\lambda}$ and $p_m>16^{m\lambda^{-1}}$, we get $3^m<\eta(p_m)^{\log_{16}{3}+0.6}. $ This implies that the interval $[\eta(p_m), \eta(p_m)+\eta(p_m)^{\log_{16}{3}+0.6}]$ contains at least $m$ prime numbers. Therefore, assuming that the $q_1, q_2, \ldots, q_{m-1}$ are already chosen,  one can pick the prime $q_m$ from the interval $[\eta(p_m), \eta(p_m)+\eta(p_m)^{\log_{16}{3}+0.6}]$ so that it is different from the first $m-1$ chosen ones.  Therefore, since $\log_{16}{3}+0.6<1$, for $\mu = \log_{16}{3}+0.6$ the assertion of the lemma holds.
\end{proof}

\begin{props}
    \label{proposition-main-on-robust-FRG}
    Let $f: \nn \rightarrow \nn$ be such that for some $0<\lambda <1$, we have $f(n)> n^{\lambda(2n+1)+1}$, and let $\eta$ be defined as above. Then given the conditions on parameters of $G_{\calA}$ stated in this subsection, 

$$\rf_{G_{\calA}}(n) \simeq f(n).$$
\end{props}
\begin{proof}
  First of all, note that $supp(F_{p_i}) = \{b_i\}$ for all $i\in \nn$, because $a_{p_i}=b_i$. Therefore, by Proposition \ref{props-RFG-for-Fp}, for $g=F_{p_i}$, we have
  \begin{align*}
    \rf_g = 2p_i q_i^{2p_i+1}.
  \end{align*}
  Consequently, since by Lemma \ref{lem-metric-2} we have $|g| = \Theta(p_i)$, we get
  \begin{align}
  \label{eq-ineq-1}
      \rf(p_i) \succeq f(p_i).
  \end{align}
  On the other hand, since by the condition \eqref{eq-aux-6.2} we have $\eta(p_i) > p_i^{\lambda}$ and by the condition in Lemma \ref{lem-on-q-s} we have $\eta(p_i)\leq q_i \leq \eta(p_i)(1+\eta(p_i)^{\mu-1})$, we get:

  \begin{align*}
  f(p_i) = p_i \eta(p_i)^{2p_i+1} \leq 2p_i q_i^{2p_i+1}
  &= \rf_g \leq 2p_i (\eta(p_i) (1+\eta(p_i)^{\mu-1}))^{2p_i+1} \nonumber \\ &= 2f(p_i) (1+p_i^{\lambda(\mu-1)})^{2p_i+1} = \Theta(f(p_i)),
  \end{align*}
  where the last equality follows from the observation that as $i \rightarrow \infty$, $(1+\eta(p_i)^{\mu-1}))^{2p_i+1}$ converges to the constant $\exp{\frac{3}{\lambda(1-\mu)}}.$ Thus we have
  \begin{align}
      \label{eq-asympt-main} 
       \rf_{g} = \Theta(f(p_i)).
  \end{align}

Now, assume that $h \in \calF' \setminus \{1\}$ is such that $|h|\leq p_i$. Then, there exists $1\leq k\leq p_i$ such that $b_k \in supp(h)$, because, otherwise, Lemma \ref{lem-metric-2} would imply that $|h| \succ p_i$. Then, by Lemma \ref{lem-metric-new}, $h\notin N_{q_i, 2p_i, p_i}$. Therefore, by \eqref{eq-asympt-main} and Proposition \ref{props-index-of-N_q,n,r}, $\rf_h\leq [G_{\calA}: N_{q_i, 2p_i, p_i}] = 2p_i q_i^{2p_i+1}=\Theta(f(p_i)), $ which implies
  \begin{align}
  \label{eq-aux-6-3-1}
  \rf_h\preceq f(p_i).
  \end{align}
  Since $h\in \calF' \setminus \{1\}$ is an arbitrary element, combining \eqref{eq-asympt-main} with \eqref{eq-aux-6-3-1}, we obtain that 
  $$\rf(p_i) \simeq f(p_i).$$

  Finally, since $f$ is non-decreasing and, by the construction of $p_i$-s, for any $n>1$, there exists $i\in \nn$ such that $p_i \leq n < p_{i+1}$ and  since $p_{i+1}/p_i$ is bounded by a constant (see \eqref{eq-p-s}), we get that $$\rf(n) \simeq f(n).$$
\end{proof}
From Proposition \ref{proposition-main-on-robust-FRG}, we immediately obtain the statement of Theorem \ref{thm-RFG-spectra}. That is, we have:
\begin{conclusion*}
    If $f: \nn \rightarrow \nn$ is non-decreasing and such that $f(n)\succeq \exp(\varepsilon n \log n)$ for some $\varepsilon>0$, then there exists a residually finite solvable group $G_f$ of derived length $3$ such that $$\rf_{G_f}(n) \simeq f(n).$$
\end{conclusion*}
~\\

\section{Computability properties}

\subsection{Computability of RFG depends on the choice of the generating set: Proof of Theorem \ref{thm-answer-to-Rauzy-1} }

\label{subsection-Rauzy_1}
\begin{thm}[Theorem \ref{thm-answer-to-Rauzy-1}]
    \label{thm-answer-to-Rauzy-1-COPY} 
    There exists a residually finite two-generated solvable group $G_{c}$ of derived length $3$ such that the computability property of its residual finiteness growth function depends on the choice of the generating set. 
\end{thm}
\begin{proof}
    We are going to describe a finitely generated residually finite group $H$ that has two different finite generating sets $X$ and $Y$ such that $\rf_{H, X}$ is computable, while $\rf_{H, Y}$ is not.

    The group $H$ will be a direct product of the groups $G_{\calA_1}$ and $G_{\calA_2}$, where $G_{\calA_1}$ and $G_{\calA_2}$ are adaptations of the construction of $G_{\calA}$ from Section \ref{section-robustness-torsion_case}. 

    Fix a non-recursive set $\calN=\{n_1, n_2, \ldots \} \subset \nn$. Let $\calP=\{p_1, p_2, \ldots \} \subset \nn$ be the set of all odd primes enumerated in their natural order.

    Then, define $G_{\calA_1}$ as the adaptation of the $G_{\calA}$ such that  the order $q_i$ of $F_{p_i}$, $i\in \nn$, is defined as 
    \[
    \begin{aligned} q_i = \begin{cases}
                       p_i    & \mbox{if $i \in \calN $ is odd ,}\\
                       p_{i+1}  & \mbox{if $i \in \nn \setminus \calN$.}               \end{cases}
        \end{aligned}
    \]
    Then, by Proposition \ref{props-RFG-for-Fp}, the residual finiteness depth of the element $F_{p_i}$ is equal to $2p_i p_i^{2p_i+1}$ or $2p_i p_i^{2p_{i+1}+1}$ depending on whether or not $i \in \calN$. And, applying Lemma \ref{lem-metric-2}, for $S_1: = \{F, s \}$ of $G_{\calA_1}$, we get $\rf_{G_{\calA_1}, S_1}(4p_i+4) $ is equal to $2p_i p_i^{2p_i+1}$ or $2p_i p_i^{2p_{i+1}+1}$ depending on  whether or not $i \in \calN$.
    This implies that for the generating set $Z_1: = \{F, s \}$ of $G_{\calA_1}$,  $\rf_{G_{\calA_1}, S_1}$ is not computable. However, by Proposition \ref{proposition-main-on-robust-FRG}, we get $\rf_{G_{\calA_1}, S_1}(n)  \simeq n^n$. 

    Now, define $G_{\calA_2}$ in complete similarity with $G_{\calA_1}$ with the only difference that for $G_{\calA_2}$, $q_i$ is always equal to $p_i$, regardless of the value of $i\in \nn$. Then, with respect to the standard generating set $Z_2=\{F, s\}$ of $G_{\calA_2}$,  $\rf_{G_{\calA_2}, Z_2}(n) \leq  \rf_{G_{\calA_1},Z_1}(n) $ for $n\in \nn$. On the other hand $\rf_{G_{\calA_2}, Z_2}(n) \simeq  \rf_{G_{\calA_1},Z_1}(n) $. Note that RFG of $H=G_{\calA_1} \times G_{\calA_2}$ with respect to the generating set $Y:=Z_1 \times Z_2$ is equal to the maximum of $\rf_{G_{\calA_1},Z_1}$ and $\rf_{G_{\calA_2},Z_2}$ at any point. Therefore, with respect to the generating set $Z_1 \times Z_2$, RFG of $H$ is not computable.  However, since $\rf_{G_{\calA_2}, Z_2}(n) \simeq  \rf_{G_{\calA_1},Z_1}(n) $, one can choose a new generating set $Z_1'$ for $G_{\calA_1}$ so that $\rf_{G_{\calA_2}, Z_2}$ dominates over $\rf_{G_{\calA_1}, Z'_1}$, which means that the RFG of $H$ with respect to the generating set $X:= Z_1' \times Z_2$ is computable, as $\rf_{G_{\calA_2}, Z_2}$ is computable. Therefore, as a group $G_c$ from the statement of Theorem \ref{thm-answer-to-Rauzy-1-COPY} we can pick the group $H$.
\end{proof}

\subsection{The Word Problem and the Individual Residual Finiteness Depth function}

Let $G=\langle X \rangle$ be a residually finite group together with its finite generating set $X$. Let us define
\begin{align*}
    \rho: (X \cup X^{-1})^* \rightarrow \nn
\end{align*}
as follows: for $w \in (X \cup X^{-1})^*$, $\rho(w)=1$ if and only if $w=_g1$. Otherwise, $\rho(w)$ is the smallest possible size of the finite quotient of $G$ in which the element represented by $w$ is not trivial. We call $ \rho: (X \cup X^{-1})^* \rightarrow \nn$ the {individual residual finiteness depth function} of $G$ with respect to $X$.

\begin{df}
    \label{def-left-partial-recursive}
    Let $f: S \rightarrow \nn$, where $S$ is a recursive (countable) set. Then, we say that $f$ is {left-computable } if the set 
    \[
    \{ (s,k)\in S \times \nn \mid f(s)>k \}.
    \]
    is recursively enumerable. Or, in other words, it means that there is a (partial) Turing machine that verifies the property $f(s)>k $ for $(s,k)\in S \times \nn$.
\end{df}

The next theorem is a characterization of the decidability of the word problem in a finitely generated residually finite group by its individual residual finiteness depth function.

\begin{thm}[Theorem \ref{thm-characterization-WP-decidability-COPY}]
    \label{thm-characterization-WP-decidability-COPY}
Let $G=\langle X \rangle$ be a residually finite with a recursively enumerable presentation along with its finite generating set $X$. Then the word problem in $G$ is decidable if and only if the individual residual finiteness depth function of $G$ is left-computable.
\end{thm}
\begin{proof}
    First, let us show the `if' part. It is an easy folklore fact that $G$ has a recursive presentation  if and only if there is a (partial) algorithm that detects precisely those  $w \in (X \cup X^{-1})^*$ that represent the trivial element of $G$. This means that if $\rho(w)\not > 1$, then it can be algorithmically verified. On the other hand, left-computableness of $\rho$ implies that in case $\rho(w) > 1$, then it can be algorithmically verified as well. Therefore, whether or not $\rho(w)> 1$ is an algorithmic property, which means that the word problem is decidable under these assumptions.

    Now, let us turn to the `only if' part. Assume that the word problem in $G$ is decidable. Then, it follows by the definition that the property whether or not $\rho(w) > 1$ is an algorithmic one, where $w \in (X \cup X^{-1})^*$. Thus what is left to show is this: assume that $w$ is such that $\rho(w)>k$ for some fixed $k>1$, then find an algorithm that will detect it.

    Next we show that there is algorithm detecting the property $\rho(w)>k$ for $w\in (X \cup X^{-1})^*$, $k>1$. Indeed, let us consider all finite groups $H_1, H_2, \ldots$ of size at most $k$ (it is well known that these groups can be algorithmically found for any $k$). Next, consider all maps $f: G \rightarrow H_i$, $i=1,2, \ldots$ such that $f(w)\neq_{H_i} 1$. (Since the word problem in $G$ is decidable, the latter condition is algorithmic.) Now, note that $\rho(w)>k$ precisely means that no such map $f: G \rightarrow H_i$ is a group homomorphism. The latter property means that for each such $f: G \rightarrow H_i$  there is $w\in (X \cup X^{-1})^*$ such that $w=_G 1$ but $f(w)\neq 1$, which can be algorithmically verified, as $G$ has decidable word problem (hence, has a recursively enumerable presentation.) Thus the proof is completed.
    \end{proof}
    \begin{rem}
        The algorithm described in the proof of Theorem \ref{thm-characterization-WP-decidability} essentially is due to McKinsey \cite{McKinsey_1943} and sometimes is called McKinsey algorithm, see  \cite{Rauzy_quotients}. 
    \end{rem}

\subsection{The RFG functions for RF groups that have decidable word problem: Proof of Theorem \ref{thm-on-RFG-for-WP} }
Let $f: \nn \rightarrow \nn$ be such that 
\begin{itemize}
\item $f$ is left-computable,
    \item $f(n)$ is non-decreasing and $f(n)>\exp(n\log(n))$.
\end{itemize}
\begin{thm}[Theorem \ref{thm-on-RFG-for-WP}]
     \label{thm-on-RFG-for-WP-COPY}
    For any $f: \nn \rightarrow \nn$ satisfying the above conditions, there exists a residually finite group $G_f$ with decidable word problem such that $f(n) \preceq \rf_{G_f}(n) \preceq (f(n))^6$ with respect to some generating set.
\end{thm}
In the following, we present a constructive proof of Theorem \ref{thm-on-RFG-for-WP-COPY}.  We will consider a fixed $f: \nn \rightarrow \nn$ that satisfies the above conditions.\\

Let $\tau: \nn \rightarrow \{(n, k) \mid f(n)>k \}$ be a fixed computable one-to-one function (the existence of which follows from the left-computableness of $F$). Denote
$$ (n_i, k_i):=\tau(i).$$

Let us fix a Turing machine $\calT$ that, for an input $l\in \nn$, lists $(n_1, k_1), \ldots, (n_l, k_l)$. Define $T_l$ as the number of steps $\calT$ requires before stopping for input $l\in \nn$. Note that $\{T_1, T_2, \ldots \}$ is recursively enumerated. Note that without loss of generality, we can assume that $T_i > n_i^{k_i}$, because, otherwise, the machine can simply make dummy steps in order to achieve this estimate.

Let $\calP$ and $\calQ$ be the sets of odd prime numbers with odd and even indices (with respect to the natural ordering), respectively; so that $\calP \cap \calQ = \emptyset$.\\

Let $\calQ =\{q_1 < q_2 < q_3 < \ldots \}$. Note that the mentioned enumeration of $\calQ$ is computable (i.e. recursive).

Let us define $\eta: \nn \rightarrow \nn$ such that, for $n\in \nn$,
\[
\eta(n) \text{~is the smallest positive integer such that $n^{\eta(n)n3^{\eta{(n)}}}$ is larger than $f(n)$.}
\]

Let $\calB: = \langle b_1, b_2, \ldots \mid [b_i, b_j]=1, \calR \rangle$, where $\calR$ is composed of the following relators:
\begin{enumerate}
    \item $b_p^p=1$ for all $p \in \calP$;
    \item $b_{p3^k}^{p^k}=1$ for all $p \in \calP$ and $k\in \nn$;
    \item $b_n=1$ if $n$ is not of one of the forms: ($n=p3^k$ for $p \in \calP$) or ($n=q$ for $q\in \calQ$);
    \item Let $l\in \nn$ and $p \in \calP$ be such that $f(p)>l$.  Denote $i=\tau^{-1}((p, l))$. Then:
    $$b_{p}=b_{p3^k}^{p^{k-1}}b_{q_{T_i}}^{p^{k-1}},$$
    where $k = \eta(l)$. 
\end{enumerate}
Let $\calA = \langle a_1, a_2, \ldots \rangle \leq \calB$ be defined so that for odd $m$, $b_m$ is $(2m, m)$-periodic with respect to $\calA$. In other words, if $n \in \nn$ is odd, $a_n = \prod_{m | n} b_m.$ \\
In more formal terms, $\calA$ can be described as
\begin{align*}
 \label{eq-calA-1}
     \calA= \langle a_i, i\in \zz \mid &[a_i, a_j]=1, i, j \in \nn; \nonumber\\
      &a_{2n}=1;\nonumber\\
     & a_i=a_{-i}, i \in \zz; \nonumber \\
      &a_{2n+1}=\prod_{\substack{m\mid 2n+1\nonumber\\
     m \text{~is odd}}} b_m, n\in \zz \rangle < \calB. \nonumber
 \end{align*}
~\\
\begin{claim}
\label{claim-computability-of-calB-WP/RFG}
    The above mentioned presentation of $\calB$ is computable.
\end{claim}
\begin{proof}
    From the definition of the presentation of $\calB$ it is evident that its computability reduces to the recursiveness of the set of relators (1)-(4). The recursiveness of the set of relators (1)-(3) is obvious. For the set of relators (4) note that it is equivalent to the recursiveness of the set of triples $\{(p, q_{T_i}, k)\}$, which is equivalent to the recursiveness of the set of triples  $\{(p, T_i, k) \mid n \in \nn \}$.  The letter one is recursive due to the definition of $T_i$-s and the fact that  $\eta$ is computable. 
\end{proof}

Combining Claim \ref{claim-computability-of-calB-WP/RFG} with Theorem \ref{thm-characterization-WP-decidability}, we get the following.
\begin{cor}
    \label{cor-WP_for_WP/RFG}
    The word problem in $G_{\calA}$ is decidable.
\end{cor}

\subsubsection{The normal form $\nu_1$.}
Note that from the definition of $\calB$ it follows that any element $b \in \calB\setminus \{1\}$ can be \emph{uniquely} decomposed in the form
\begin{align*}
     b = b_{n_1}^{l_1} \ldots b_{n_k}^{l_k}, \text{~where } n_1 < \ldots < n_k, ~ b_{n_i}^{l_i}\neq 1, i=1, 2, \ldots, k,
\end{align*}
and $n_i \in \{p3^k, q \mid p \in \calP, q \in \calQ, k \in \nn \cup \{0\} \}$ and such that:\\
if $n_i = p \in \calP$, then $0 < l_i < p$;\\
if $n_i = p 3^k$, $k\in \nn$, and $b_{p 3^k}$ does not participate in a relator of the form (4), then $0 < l_i < p^{k}$;\\
if $n_i = p 3^k$, $k\in \nn$, and $b_{p 3^k}$ participates in a relator of the form (4), then $0 < l_i < p^{k-1}$;\\
if $n_i = q_{T_j} \in \calQ$, then $0< l_i < p^{k}$, where $p, k$ are as in (4) above.\\

\noindent
For the case $n_i = q_{T_i} \in \calQ$, to make things more clear, let us notice that, by definition, for each $q \in \calQ$, $b_q$ appears in a unique relator of the form (4). Also, notice that from relators (1)-(4) it follows that the order of $b_{q_{T_i}}$ in $\calB$ is $p^k$. 

Let us denote the above described normal form of $\calB$ by $\nu_1$. Note that with respect to the normal for $\nu_1$, the above described presentation of $\calB$ satisfies Condition (*) (see Definition \ref{def-condition-(*)}). For example, to check the part (i) of Condition (*) for element $b_{p3^{k_0}}$, ${k_0}\in \nn$, one can notice that if ${k_0}$ is small enough so that $b_{p3^{k_0}}$ appears in relators of the type (4), then one can take all $k\in \nn \cup \{0\}$ for which $b_{p3^{k}}$ appear in such relators and all $q \in \calQ$ that also appear in those relators, then $\calB$ is a direct product of the subgroup of $\calB$ generated by those elements with another subgroup. To check the part (ii) of Condition (*), simply note that all elements of $\calB$ are of finite order.

Combining the above observation about Condition (*) of the presentation of $\calB$ with Theorem \ref{thm-a_general_case_for_G_A_being_RF}, we immediately get the following.
\begin{lem}
    $G_{\calA}$ is a residually finite group. \qed
\end{lem}
The next easy property, along with Property \ref{prop-easy-2}, will be helpful for proving Claim \ref{claim-about-upper-bound-RFG-WP}.
\begin{propty}
    \label{prop-easy-1}
    Let $h_1, h_2 \in \calF'$ be such that $|h_1|_{n, \nu_1}=|h_1|_{n, \nu_1}=0$ for all $n$ such that $b_n$ appears in a relator of the type (4) and is of the form $p3^k$, $k\in \nn\cup \{0\}$, and $p\in \calP$ is fixed. Then, for those $n$, $|h_1h_2|_{n, \nu_1}=0$ and $|h_1^{-1}|_{n, \nu_1}=0$. \qed
\end{propty}
\subsubsection{The normal form $\nu_2$.}
Note that, in analogy with the case of the normal form $\nu_1$, from the definition of $\calB$ it follows that any element $b \in \calB\setminus \{1\}$ can be \emph{uniquely} decomposed in the form
\begin{align*}
     b = b_{n_1}^{l_1} \ldots b_{n_k}^{l_k}, \text{~where } n_1 < \ldots < n_k, ~ b_{n_i}^{l_i}\neq 1, i=1, 2, \ldots, k,
\end{align*}
and $n_i \in \{p3^k, q \mid p \in \calP, q \in \calQ, k \in \nn \cup \{0\} \}$ and such that:\\
if $n_i = p \in \calP$, then $0 < l_i < p$;\\
if $n_i = p 3^k$, $k\in \nn$, then $0 < l_i < p^{k}$;\\
if $n_i = q_{T_j} \in \calQ$, then $0< l_i < p^{k-1}$, where $p, k$ are as in (4) above.\\

The next easy property complements Property \ref{prop-easy-1} in the proof of Claim \ref{claim-about-upper-bound-RFG-WP}.
\begin{propty}
    \label{prop-easy-2}
    Let $h_1, h_2 \in \calF'$ be such that $|h_1|_{n, \nu_2}=|h_1|_{n, \nu_2}=0$ for all $n$ such that either it does not appear in any relator of the type (4)  or $n \in \calQ$. Then $|h_1h_2|_{n, \nu_2}=0$ and $|h_1^{-1}|_{n, \nu_2}=0$. \qed
\end{propty}
\subsubsection{The lower and upper bounds for $\rf_{G_{\calA}}$}
\begin{claim}
    \label{claim-about-lower-bound-RFG-WP}
    For all $p\in \calP$, $\rf_{F_{p}}\geq p^{2kp3^k}$, where $k=\eta(l)$ and $l$ is the largest integer such that $f(p)>l$. 
\end{claim}
\begin{proof}
    Let $N \triangleleft G_{\calA}$ be such that $F_{p} \notin N$ and $[G_{\calA}:N]< \infty$. Let $h \in G_{\calA}$ be such that $\vec{\phi}_{p^k, 2p3^k}(h) \neq \vec{0}$ and the $\calF$-part of $h$ is a product of elements of the form $F^{s^r}$, $0\leq r \leq 2p3^k-1$. To prove the claim, it would suffice to prove that $h \notin N$, then the proof will follow from the Pigeonhole Principle.

    By the method of contradiction, let us assume that $h\in N$. Let $0\leq j_0 \leq 2p3^k-1$ be the smallest index among those $0\leq j \leq 2p3^k-1$ for which ${\phi}_{p^k, 2p3^k}(h)$ is a multiple of the smallest power of $p$.  Also, let $i = \tau^{-1}((p, l))$. Then, take 
    \[
    h': = [F^{s^{2q_{T_i}p3^k+j_0}}, h].
    \]
    Note that due to Proposition \ref{prop-key-identities}, $h'$ can be decomposed as a product of elements of the form $F_m$ for odd $m$-s, such that the sum of exponents of elements $F_m$ that are equal to each other $\mod 2p3^k$ is the multiple of the smallest power of $p$ when $m \equiv 0 \mod 2p3^k$. Moreover, no $m$ for this decomposition is a multiple of $p3^{k+1}$. This means, thanks to relators (4) described above and the observation that the order of elements $F_m3^r$ is $P^r$, where $m$-s are odd and $P$ is the product of all distinct prime divisors from $\calP$ of $m$-s, that a sufficient power of $h'$ will be equal to $F_p$, implying that $F_p \in N$. A contradiction.
\end{proof}
\begin{claim}
  \label{claim-about-upper-bound-RFG-WP}
     $ \rf_{G_{\calA}}(n) \preceq f(n)^6.$
\end{claim}
\begin{proof}
   Let $g \in G_{\calA} \setminus \{1\}.$ Due to Proposition \ref{prop-RFG-TypeI_and_II_cases}, without loss of generality we can assume that $g \in \calF'$. Let $g$ be in correspondence with $a \in \calA$ due to the isomorphism from Proposition \ref{prop-about-calF'-calA-isom}. Let $n_1 \in \nn$ be the smallest positive integer such that $a \in \langle b_1, \ldots, b_{n_1-1}, b_{n_1} \rangle \setminus \langle b_1, \ldots, b_{n_1-1} \rangle.$ Then, note that by Lemma \ref{lem-metric-new}, the word-metric length of $g$, $|g|$, satisfies $|g| \succeq n_1$.\\
   \noindent
   \underline{Case 1:} ($b_{n_1}$ is not participating in relators of the type (4).)\\
   In this case, define 
   \begin{align*}
           N_g: = \{ h \in G_{\calA} \mid 2{n_1}& \text{~divides~}|h|_s, \\
    &2|g|_{b_{n_1}, \nu_2} \text{~divides~} |h|_{b_{n_1}, \nu_2},\\ 
    & \vec{\phi}_{2|g|_{b_{n_1}, \nu_2}, 2n_1}(h) = \vec 0_{\mathbb{Z}_{2|g|_{b_{n_1}, \nu_2}}^{2n_1}} \}. 
    \end{align*}
    Then, $g\notin N_g \triangleleft G_{\calA}$ and $[G_{\calA} : N_g] \leq (2n_1)(2l_1) (2l_1)^{2n_1} \preceq |g|^{|g|}$, where $l_1: = |g|_{b_{n_1}, \nu_2}.$ The last inequality follows from Lemma \ref{lem-metric-2}. The fact that $N_g$ is a group follows, in particular, from Property \ref{prop-easy-2}. The fact that $N_g$ is a normal subgroup, follows, in particular, from the condition that if $h_1, h_2 \in \calF$, $|h_1|_{b_{n_1}, \nu_2}  \equiv 0 \mod 2|g|_{b_{n_1}, \nu_2}$, then  $|h_1|_{b_{n_1}, \nu_2}  \equiv |h_1 h_2|_{b_{n_1}, \nu_2}\mod 2|g|_{b_{n_1}, \nu_2}$. \\
    \noindent
   \underline{Case 2:} ($b_{n_1}$ participates in a relator of the type (4) and is of the form $p3^k$ for $p \in \calP$, $k \in \nn \cup \{0\}$.)\\
   Let $K\in \nn$ be the largest integer such that $b_{p3^K}$ participates in a relator of the form (4).
    In this case, define 
   \begin{align*}
           N_g: = \{ h \in G_{\calA} \mid 2p3^K& \text{~divides~}|h|_s, \\
    &|h|_{b_{n_1}, \nu_1}=0,\\ 
    & \vec{\phi}_{2|g|_{2p3^K}, p^K}(h) = \vec 0_{\mathbb{Z}} \}. 
    \end{align*}
    Then, $g\notin N_g \triangleleft G_{\calA}$ and $[G_{\calA} : N_g] \leq (2p3^K)(p^K) (p^K)^{2p3^K}  \preceq f(p)^6 \leq f(|g|)^6.$ For the last two inequalities we used the definition of $\eta$ that participates in the definition of relators (4) and Lemma \ref{lem-metric-2}, which in particular for the length of $g$ states that $|g| \succeq p$. \\
    \noindent
    \underline{Case 3:} ($b_{n_1} = q_{T_i},$ where $q_{T_i}$ is as in relator  (4).)\\
    In this case, define 
   \begin{align*}
           N_g: = \{ h \in G_{\calA} \mid 2q_{T_i}& \text{~divides~}|h|_s, \\
    &|h|_{q_{T_i}, \nu_2}=0,\\ 
    & \vec{\phi}_{2q_{T_i}, p^k}(h) = \vec 0 \}. 
    \end{align*}
    Then, $g\notin N_g \triangleleft G_{\calA}$ and $[G_{\calA} : N_g] \leq (2q_{T_i})(p^k) (2q_{T_i})^{p^k} \preceq q_{T_i}^{q_{T_i}} \preceq |g|^{|g|}.$ For the last two inequalities we used the condition that $T_i>p^k$ and Lemma \ref{lem-metric-2}, which in particular implies $|g| \succeq q_{T_i}.$\\

\end{proof}

\begin{cor}
    \label{cor-for_RFG_WP} 
   $ f(n) \preceq \rf_{G_{\calA}}(n) \preceq f(n)^6$.
\end{cor}
\begin{proof}
    The left-hand side inequality follows from Claim \ref{claim-about-lower-bound-RFG-WP}, the fact that the set $\calP \subset \nn$ is dense (by Bertrand's postulate), and Lemma \ref{lem-metric-2}, which asserts that $|F_p| = \Theta(p)$. The upper bound for $\rf_{G_f}(n)$ is from Claim \ref{claim-about-upper-bound-RFG-WP}.
\end{proof}
\begin{conclusion*}
    To conclude the proof of Theorem \ref{thm-on-RFG-for-WP-COPY}, we can simply take $G_f = G_{\calA}$.
\end{conclusion*}


\subsection{Obstructions for Higman's embedding theorem for RF groups with decidable WP, and a generalization: Proof of Theorem \ref{thm-turing-dominant}}

\subsubsection{Residually finite groups with decidable word problem and Turing dominant RFG: Proof of Theorem \ref{thm-turing-dominant} }

\label{subsec-turing-dominant}

 \begin{thm}[Theorem \ref{thm-turing-dominant}]
\label{thm-turing-dominant-COPY}
    For any r.e. Turing degree $\sigma$ there exists a RF group $G_{\sigma}$ with decidable word problem and {$\sigma$-dominant} RFG function up to the standard equivalence. 
\end{thm}

In the following, we present the proof of Theorem \ref{thm-turing-dominant-COPY}.

\begin{lem}
\label{lem-aux-turing-dominant-map}
     For each r.e. Turing degree $\sigma$, there exists a $\sigma$-dominant non-decreasing map $g: \nn \rightarrow \nn$ such that it is left-computable .
 \end{lem}
 \begin{proof}

Let $\calN = \{n_1, n_2, \ldots \} \subseteq \nn$ be a recursively enumerated set of Turing degree $\sigma$.

Let us define $g: \nn \rightarrow \nn$ as $g(n) = \max \{ i \mid n_i\leq n \}$. Then, clearly $g(n)$ is non-decreasing and of Turing degree not exceeding $\sigma$. Now, assume that $m\geq g(n)$ is given. Then, to answer the question whether or not $n \in \calN$, one can simply algorithmically list all elements $n_1, n_2, \ldots, n_m$ and check if $n$ is among them. This observation implies that $g$ is $\sigma$-dominant. 

Finally, let us show that $g$ is left-computable. Assume that $(n, k) \in \nn \times \nn$ is such that $f(n)>k$. We want to verify that algorithmically. To that end, we can algorithmically enlist $n_1, n_2, \ldots, n_k, n_{k+1}$. Then, the condition  $f(n)>k$ is equivalent to saying that $n \geq \max\{n_1, n_2, \ldots, n_{k+1}\}$, which is of course algorithmically verifiable. Thus $g$ is left-computable, implying that the lemma is proved.
     \end{proof}

     Now, to complete the proof of Theorem \ref{thm-turing-dominant-COPY}, let us consider $f(n): = 2^{2ng(n)}$. Then, easy to see that $f(n)$ is $\sigma$-complete and satisfies all the assumptions of Theorem \ref{thm-on-RFG-for-WP}. Therefore, due to Theorem \ref{thm-on-RFG-for-WP}, there exists a finitely generated residually finite group $G_f$ with decidable word problem and residual finiteness growth function $\rf_{G_f}$ satisfying $f(n) \preceq \rf_{G_f}(n) \preceq (f(n))^6$. Finally, in order to complete the proof of Theorem \ref{thm-turing-dominant}, it is enough to note that $f(n) \simeq (f(n))^6$, which implies that $G_f$ satisfies the assertion of the theorem.

\section{Conjugacy separable groups with decidable word problem and undecidable conjugacy problem: Proof of Theorem \ref{thm-ashot-s_question}}
\label{sec-ashot-s_question}
The goal of this section is to prove the following theorem.
\begin{thm}[Theorem \ref{thm-ashot-s_question}]
\label{thm-ashot-s_question-COPY}
There exists a finitely generated conjugacy separable group with a decidable word problem and an undecidable conjugacy problem.
\end{thm}

Let us consider $\calA < \calB$ such that $\calB = \langle b_1, b_2, \ldots \rangle$ and $|b_i|=4$ if $3 \not| i$, otherwise, if $3 | i$, $|b_i|=2$. Let $\calN = \{n_1, n_2, \ldots \} \subset 3\nn$ be a recursively enumerable but not recursive set of integers that are multiples of $3$. Let $$\calP = \{ p_1, p_2, \ldots \}$$
be a recursive set of prime numbers in their natural order such that for infinitely many $ k \in \nn$,
\begin{align}
    \label{condition - 0000}
p_{k+1}>4^kp_1p_2\ldots p_k.
\end{align}
Also, for $k=1,2,\ldots$, we require $$p_{3k+2}-p_{3k+1}=K,$$ 
where $K$ is a constant independent of the choice of $k$. The existence of such a set $\calP$ for some $K\in \nn$ follows from the celebrated result of Zhang \cite{zhang-twin_primes}. Note that such $K$ is always even.
    
  We define $\calA$ so that for $m, n \in \nn$, $b_m \in supp(a_n)$ if $n \equiv p_m \mod 2p_m.$ Besides the commuting relators and the relators $b_i^4=1$ if $3 \not| i$, otherwise, if $3 | i$, $b_i^2=1$, we require the elements of $\calB$ to satisfy also the following relators $$ b_{{n_i}}= b_{3i+1}^2 b_{{3i+2}}^2$$ for $i=1, 2, \ldots,$
    which implies $$a_{p_{n_i}}= a_{p_{3i+1}}^2a_{p_{3i+2}}^2,$$ and 
    $$F_{p_{n_i}}= F_{p_{3i+1}}F_{p_{3i+2}}.$$
    More formally, we define $\calB$ and $\calA$ as follows.
   \begin{align}
   \label{pres-B-CS}
     \calB= \langle b_1, b_2, \ldots \mid & ~[b_i, b_j]=1, \nonumber\\
     & b_{{n_i}}= b_{3i+1}^2 b_{{3i+2}}^2, i=1, 2, 3, \ldots ,\\
     & b_{3i+1}^4=b_{3i+2}^4=1, i = 1, 2, 3, \ldots,\nonumber\\
     & b_{3i}^2=1,  i=1, 2, 3, \ldots \rangle
 \end{align}
 \begin{align*}
     \calA= \langle a_i, i\in \zz \mid &[a_i, a_j]=1, i, j \in \nn; \nonumber\\
      &a_{2n}=1,\nonumber\\
     & a_i=a_{-i}, i \in \zz; \nonumber \\
    &a_{p_{i}}=b_{i}, i \in \nn\\
      &a_{2n+1}=\prod_{\substack{p\mid 2n+1\nonumber\\
     p \in \calP}} a_p, n\in \zz \rangle < \calB. \nonumber
 \end{align*}
\begin{lem}
      \label{lem-WP-for-WP/CP}
      The group $G_{\calA}$ has decidable word problem.
  \end{lem}
  \begin{proof}
      One can easly see that the presentation \eqref{pres-B-CS} of $\calB$ is computable, because the word problem for $\calB$ with respect to that presentation straightforwardly reduced down to the detection of the relators $b_{{n_i}}= b_{3i+1}^2 b_{{3i+2}}^2$, $b_{3i+1}^4=b_{3i+2}^4=1$, $b_{3i}^2=1$, which form a recursive set, as $\{(n_i, i) \mid i \in \nn \}$ is recursive. Therefore, the presentation of $\calA$ is computable, and hence the decidability of the word problem in $G_{\calA}$ follows from Theorem \ref{thm-WP-in-G_A}.
  \end{proof}

  To prove the conjugacy separability of $G_{\calA}$, let us first define the normal form $\nu: \calB \rightarrow \{b_1^{\pm 1}, b_2^{\pm 1}, \ldots \}^{*}$ as the one in which for all $b\in \calB \setminus \{1\}$, $\nu(g)$ is of the form $b_{i_1}^{k_1}\ldots b_{i_l}^{k_l}$, $i_1<i_2<\ldots < i_l$ and such that: $i_j \notin \calN$; if $i_j \equiv 1, 2 \mod 3$, then $0<k_j<4$; if $i_j \equiv 1 \mod 3$, then $0<k_j<2$. Note that it follows from the definition of $\calB$ that this indeed defines a normal form on $\calB$. Moreover, it satisfies Condition (*) defined in Section \ref{subsec-more-general-periodicity}, hence, by Theorem \ref{thm-a_general_case_for_G_A_being_RF}, $G_{\calA}$ is residually finite.

   \begin{lem}
       \label{lem-separability-CP_WP}
       The group $G_{\calA}$ is conjugacy separable.
   \end{lem}    
       \begin{proof}
        Let $g_1, g_2 \in G_{\calA}$ and $g_1 \not\simeq_{conj} g_2$ in $G_{\calA}$. We want to show that there is a finite index normal subgroup $N_{g_1, g_2}\triangleleft G_{\calA}$ such that $g_1N_{g_1, g_2}$ is not conjugate to $g_1N_{g_1, g_2}$ in $G_{\calA}/N_{g_1, g_2}$. Since $G_{\calA}/ \calF'$ is isomorphic to $\zz \wr \zz$ (see Corollary \ref{cor-reduction-to-lamplighter}) and the latter group is conjugacy separable \cite{remeslennikov-conj_sep}, without loss of generality we can assume that for some $x \in \zz$, $h=\prod_{i=1}^t (F^{s^{x_i}})^{k_i}$ (such that $x_1>\ldots>x_t$ and $k_i\neq 0$) and $f_1, f_2 \in \calF'$, we have
 $g_1=s^xhf_1$ and $g_2=s^xhf_2$.

Let $r\in \nn$ be the largest positive integer such that {$b_r \in \supp_{\nu}(f_1)\cup \supp_{\nu}(f_2)$} and no prime $p_l$ for $l\leq r$ is a divisor of $x$ if $x\neq 0$, where $\nu: \calB \rightarrow \{b_1^{\pm 1}, b_2^{\pm 1}, \ldots \}^*$ is as defined above.
   Let 
\begin{align}
  \label{condition-000}
  k> \max\{r, n_1, n_2, \ldots, n_r, |x_i-x_j| \mid 1 \leq i, j \leq t \}
\end{align}
  be such that
\begin{align}
  \label{condition-001}
   4^{k} p_1p_2 \ldots p_k > |g_1|_{\{F, s \}}, |g_2|_{\{F, s \}}
\end{align}
  and
\begin{align}
       \label{condition-01}
 p_{k+1}>4^{k} p_1p_2 \ldots p_k.
\end{align}
  Moreover, we assume that 
\begin{align}
       \label{condition-11}
\text{$4^{k} p_1p_2 \ldots p_k / \gcd(4^{k} p_1p_2 \ldots p_k, x)$ is divisible by $4$}
\end{align}
      Let us denote $T_k = 4^{k} p_1p_2 \ldots p_k$.
   Note that the existence of such a $k$ follows from the condition \eqref{condition - 0000}.  
   Let us define 
    \begin{align*}
           N_k: = \{ h \in G_{\calA} \mid T_{k}& \text{~divides~}|h|_s \\
         & \vec{\phi}_{4,T_{k}}(h) = \vec 0_{\mathbb{Z}_4^{T_{k}}} \}. 
    \end{align*}
    Note that the definition of $N_k$ is in  analogy with the definition of $N_{q,T, b_r} $ from Definition \ref{def-N_q,n,r}. Its normality as a subgroup of $G_{\calA}$ also can be shown in complete analogy with the normality of $N_{q,T, b_r} $ as in Lemma \ref{lem-normality-N_q,n,r}. (Also, note that at least one of $g_1$ and $g_2$ does not belong to $N_k$.)
   
   We will show that if $g_1$ and $g_2$ are conjugate in  the quotient $G_{\calA}/N_k$, then $g_1$ is conjugate to $g_2$ in $G_{\calA}$. 
   
Assume that for some $\bar{g} \in G_{\calA}$ we have that 
 $[\bar{g}]$ is a conjugator of $[g_1]$ with $[g_2]$ in $G/N_k$, where $[\cdot]$ denotes the coset class. Note that $(G_{\calA}/N_k) \big/ (\calF'/\calF'\cap N_k) \simeq \zz_{4} \wr \zz_{T_k}$. 

As follows from \cite[Lemma 1.3]{sale-conj_wreath_prod}, without loss of generality, we can assume that $\bar{g}$ is of the form $\bar{g} = s^y f_3 \prod_{j=0}^{T_k-1} (F^{s^j})^{\alpha_j+\beta_j}$, where $f_3 \in \calF'$ and if $x\neq 0$, $\beta_k=\beta_j$ for $k \equiv j \mod |x|$ and such that $ s^y f_3 \prod_{j=0}^{T_k-1} (F^{s^j})^{\alpha_j+\beta_j}$ is a conjugator of $g_1$ with $g_2$ in $G_{\calA}$. Moreover, since $\calF'$ is the center of $\calF$ and for $f \in \calF'$ we have $f^s =f^{-1}$, for the sake of our argument, without loss of generality, we can assume that $f_3=1$.

Further, note that since $k \equiv j \mod |x|$ implies $\beta_k = \beta_j$, because of condition \eqref{condition-11}, for each $\beta_k$ we have that the number of copies of it in $\{\beta_0, \ldots, \beta_{T_k-1} \}$ is a multiple of $4$. Since elements of $\calF'$ are of order $4$, this implies that $g_1$, when conjugated by $ s^y f_3 \prod_{j=0}^{T_k-1} (F^{s^j})^{\alpha_j+\beta_j}$, coincides with $g_1$ conjugated by $ s^y f_3 \prod_{j=0}^{T_k-1} (F^{s^j})^{\alpha_j}$. In other words, without loss of generality, we can assume that $\beta_1=\ldots =\beta_{T_k-1}=0$. This in particular means that, without loss of generality, we can assume that the conjugator $\bar{g}$ of $g_1$ with $g_2$ in $G_{\calA}/N_k$ is also a conjugator of $g_1$ with $g_2$ in $G_{\calA}/\calF'.$ Now, to finish the proof, one only needs to observe that due to the conditions \eqref{condition-000} and \eqref{condition-001}, the latter observation implies that in fact $\bar{g}$ is also a conjugator of $g_1$ with $g_2$ in $G_{\calA}$.

       \end{proof}
    \begin{lem}
        \label{lem-key-for-ashots-question}
     $F^2(F^{s^K})^2$ is conjugate to $F^2(F^{s^K})^2 F_{p_{n}}^{-1}$ if and only if $n \in \calN$.
    \end{lem}
    \begin{proof}
        To see the `if' part, let us assume that $n=n_i \in \calN$. Then, since $K$ is even, applying commutator identities and identities from Proposition \ref{prop-key-identities}, we have 
        \begin{align*}
               [F^2(F^{s^K})^2, F^{s^{p_{3i+2}}}] &= [(F^{s^K})^2, F^{s^{p_{3i+2}}}][F^2, F^{s^{p_{3i+2}}}] \\
       &= ([F, F^{s^{p_{3i+2}-K}}]^{s^K}[F, F^{s^{p_{3i+2}}}])^2 = [F, F^{s^{p_{3i+1}}}]^2[F, F^{s^{p_{3i+2}}}]^2 \\
       &= F_{p_{3i+1}}^2F_{p_{3i+2}}^2 = F_{p_{n_i}}.
        \end{align*}
        Therefore, since $F_{p_{n_i}}$ is in the center of $\calF < G_{\calA}$, we get $$F^2(F^{s^K})^2F_{p_{n_i}}^{-1} = (F^2(F^{s^K})^2)^{F^{s^{p_{3i+2}}}},$$
        hence $F^2(F^{s^K})^2$ is conjugate to $F^2(F^{s^K})^2 F_{p_{n_i}}^{-1}$ in $G_{\calA}$.\\
        ~\\
        Now let us prove the ``only if" part. Assume that $F^2(F^{s^K})^2$ is conjugate to $F^2(F^{s^K})^2 F_{p_{n}}^{-1}$. Then, for some $g\in \calF$, $l\in \zz$, we have
        $(s^lg)F^2(F^{s^K})^2 (s^lg)^{-1} = F^2(F^{s^K})^2 F_{p_{n}}^{-1}.$ This implies that
        \[
        [s^lg, F^2(F^{s^K})^2] = [g, F^2(F^{s^K})^2]^{s^l}[s^l, F^2(F^{s^K})^2] = [g, F^2(F^{s^K})^2][s^l, F^2(F^{s^K})^2]= F^{-1}_{p_n}.
        \]
        Note that, as $ [g, F^2(F^{s^K})^2], F^{-1}_{p_n} \in \calF'$, it must be that $[s^l, F^2(F^{s^K})^2] \in \calF'$. On the other hand, by Proposition \ref{prop-for-reduced-form-uniqueness}, $[s^l, F^2(F^{s^K})^2] \in \calF'$ if and only if $l=0$. Thus we get $l=0$. Therefore, applying Proposition \ref{prop-key-identities}, we get 
        \begin{align*}        
            [g, F^2(F^{s^K})^2] = [g, F]^2 [g, F^{s^K}]^2 = F_{p_n}^{-1}.
        \end{align*}
        The last equality implies that $F_{p_n}$ can be represented as a product of squares of elements from $\calF'$, which, by Proposition \ref{prop-about-calF'-calA-isom}, implies that $a_{p_n}$ can be represented as product of squares of elements from $\{a_1, a_2, \ldots \}$, which is true only for $n \in \calN$ by the definition of $\calA =\langle a_1, a_2, \ldots \rangle.$ Thus the lemma is proved. 
    \end{proof}
\begin{conclusion*}
    Lemma \ref{lem-key-for-ashots-question} immediately implies that if $\calN\subset \nn$ is chosen to be non-recursive, then the conjugacy problem in $G_{\calA}$ is undecidable, as otherwise, for each $n\in \nn$, depending on whether or not  $ F^2(F^{s^K})^2$ is conjugate to $ F^2(F^{s^K})^2 F_{p_n}^{-1}$, we could computationally conclude whether or not $n \in \calN$. On the other hand, it follows from Theorem \ref{thm-RF-follows-from-periodicity} that $G_{\calA}$ has decidable word problem.
    \end{conclusion*}
\mbox{}
\begin{acknowledgement}
    I would like to thank Ivo Herzog for his seminar talk at OSU in 2023 that sparked the early idea of this work. I would like to thank Henry Bradford for motivating discussions that served as an impetus for advancing this project. I would like to thank Markus Steenbock for hosting me in Vienna during the Summer of 2023, where an early draft of this work emerged; also for his valuable comments and for supporting me from his grant by the Austrian Science Fund (FWF) [10.55776/P35079].  I am thankful to Emmanuel Rauzy for very thoughtful critical remarks and collaboration on the appendix. I am also thankful to Barbro and Simon Fitzjohns for hosting me in their nice house and providing with conditions for efficient work after my move to Southampton. 
    The final stage of this work was mostly completed during my 6-months CARMIN fellowship at IHES. I would like to thank IHES staff, members and visitors for making my stay there productive, and to Delaram Kahrobaei, for her invaluable support. I also acknowledge the support from the Institut Henri Poincar\'e (UAR 839 CNRS-
Sorbonne Universit\'e) and LabEx CARMIN (ANR-10-LABX-59-01).
\end{acknowledgement}
\newpage
\section*{Appendix: Two non-embeddability theorems for residually finite groups.\\ By Arman Darbinyan and Emmanuel Rauzy.}

In this appendix, we use the construction presented in the body of the article to prove two results: 
\begin{enumerate}
    \item We separate the properties ``being effectively residually finite'' from ``having a residual finiteness growth function bounded from above by a computable function''. 
    \item We answer a problem of Nies by constructing a finitely generated residually finite group with co-semi-decidable word problem that cannot be realized as a group of computable permutations of $\nn$. 
\end{enumerate}
The first point is proved in Section \ref{sec-separating-eff-res}, the second point in Section \ref{sec-NiesPB}. 

\section{Separating effective residual finiteness from sub-computable residual finiteness growth}
\label{sec-separating-eff-res} 
A finitely generated group $G$ is \emph{effectively residually finite} \cite{rauzy-on_higman_embedding} if there is an algorithm that, on input a word $w$ on the generators of $G$ that defines a non-identity element of $G$, outputs a finite group $Q_w$ (given by a Cayley Table) and a morphism $\varphi_w:G\rightarrow Q_w$ (given by the images of the generators of $G$)  with $\varphi_w(w)\ne1$. \\
 In \cite{rauzy-on_higman_embedding}, the second named author of the present appendix built a residually finite group with decidable word problem that is not effectively residually finite precisely by guaranteeing that its residual finiteness growth would grow faster than all computable functions. 
The problem of actually separating ``effective residual finiteness'' from ``having sub-computable residual finiteness growth'' was thus left open in \cite{rauzy-on_higman_embedding}. 
Here, we prove the following theorem: 
\begin{thm}\label{thm-Main-appendix-thm}
There exists a residually finite group with decidable word problem whose residual finiteness growth is bounded from above by a computable function, but which is still not effectively residually finite.
\end{thm}
The converse implication is easily seen to hold:
\begin{lem}
    Any effectively residually finite group has sub-computable residual finiteness growth. 
\end{lem}
It follows from this lemma that the groups built thanks to Theorem \ref{thm-on-RFG-for-WP} that admit left-computable residual finiteness growth functions which dominates all computable functions cannot be effectively residually finite.  

We now begin the construction of the group $G_{\calA}$ that will be the witness for Theorem \ref{thm-Main-appendix-thm}. \\
  A pair $\calM, \calN \subset \nn$ of disjoint subsets of $\nn$ is called \emph{recursively enumerable and recursively inseparable} if $\calM$ and $\calN$ are recursively enumerable and there is no recursive subset of $\mathbb{N}$ that contains $\calM$ and is disjoint from $\calN$.
    
     For what follows, we fix a pair $\calM=\{m_1, m_2, \ldots \} \subset \nn, \calN=\{n_1, n_2, \ldots \} \subset \nn$ of recursively enumerated and recursively inseparable sets. 
     \begin{rem}
        For the existence of recursively enumerable and recursively inseparable pairs of subsets of natural numbers see, for example, \cite{smullyan1958undecidability}. It turns out that this concept is quite helpful in obtaining various exotic group theoretical properties of computability nature. See \cite{miller1981word_problem, darbinyan_inventiones, darbinya_JSL} for more group theoretical applications of this concept.
    \end{rem}
    In what follows, $p_n$ is the $n$-th odd prime number.  
We define $\calB$ and $\calA$ as follows.
\begin{align*}
 \calB= \langle b_1, b_2, \ldots \mid & ~[b_i, b_j]=1, \nonumber\\
     & b_{n_i}^{2^i}=1, i=1, 2, 3, \ldots ,\\
     & b_{m_i}^{3^i}=1, i=1, 2, 3, \ldots. \nonumber
 \rangle
 \end{align*}
 \begin{align}
 \label{eq-calA_rauzy_conj}
     \calA= \langle a_i, i\in \zz \mid &[a_i, a_j]=1, i, j \in \nn; \nonumber\\
      &a_{2n}=1,\nonumber\\
     & a_n=a_{-n}, \, n \in \zz, \\
    &a_{p_{n}}=b_{n},\, n \in \nn,\nonumber\\
      &a_{2n+1}=\prod_{\substack{p\mid 2n+1\nonumber\\
     p \in \calP}} a_p, n\in \zz \rangle < \calB. \nonumber
 \end{align}
 
 Remark that, thanks to the above definition:
 \begin{itemize}
     \item The order of $a_{p_n}$ is a power of $2$ if $n\in\calN$,
     \item The order of $a_{p_n}$ is a power of $3$ if $n\in\calM$,
     \item And $a_{p_n}$ has infinite order if $n \in \nn \setminus( \calM \cup \calN)$.
 \end{itemize}

Theorem \ref{thm-Main-appendix-thm} follows directly from the following lemmas:
\begin{lem} \label{lem-comp-bounded-growth}
     The group $G_{\calA}$ for $\calA$ given as in (\ref{eq-calA_rauzy_conj}) has residual finiteness growth bounded from above by a computable function. 
\end{lem}
\begin{lem} \label{lem-not-eff-RF}
     The group $G_{\calA}$ for $\calA$ given as in (\ref{eq-calA_rauzy_conj}) is not effectively residually finite. 
\end{lem}
\begin{lem}
  \label{lem-WP_for_GA_rauzy_conj}
    The group $G_{\calA}$ has decidable word problem.
\end{lem}

\begin{proof}[Proof of Lemma \ref{lem-comp-bounded-growth}]
     From the way the group $\calA$ is defined in \eqref{eq-calA_rauzy_conj}, it can easily be seen that $\calA$ is periodic with respect to $\calB$ (as before, the period of $b_n$ is $(2p_n,p_n)$). Furthermore, the generators of $\calB$ form a basis for $\calA$.  Therefore, by Corollary \ref{cor-RF_of_G_AND_upper_estimate_for_individual_RFG}, $G_{\calA}$ is residually finite. \\
     We now compute an upper approximation of its residual finiteness growth. 
     \\
      Let $g \in G_{\calA} \setminus \{1\}$, given as a word in $\{F^{\pm 1}, s^{\pm 1}\}^*$ of length $n\in \nn$. If $g\notin \calF'$, Proposition \ref{prop-RFG-TypeI_and_II_cases} provides the desired upper bound on $\rf_g$. We can thus assume that $g\in \calF'$.\\ 
By Proposition \ref{prop-about-calF'-calA-isom}, the correspondence $F_n \leftrightarrow a_n$ for odd $n\in \nn$ induces an isomorphism between $\calF'$ and $\calA$, which puts $g$ in correspondence with $a_g \in \calA \setminus \{1\}$. Since  $\{a_p \mid \text{~ $p$ is an odd prime} \}$ is a basis for $\calA$, there exists an odd prime $p_\star$ and $k\in \zz\setminus \{0\}$ such that $a_g = a_{p_\star}^k a_g'$, $a_{p_\star}^k \neq 1$ and $a_g' \in \langle a_p \mid p \neq p_\star \rangle$. Moreover, the pair $(p_\star, k)$ can be algorithmically computed. 
\\
There is some $n_0\in\nn$ such that $p_\star=p_{n_0}$. Three cases arise: 
\begin{itemize}
    \item If $n_0\in\calN$, then $a_{p_{n_0}}$ has order a power of $2$, and we get that $a_{p_{n_0}}\notin N_{2, 2p_0, b_{n_0}}$. 
    \item If $n_0\in\calM$, then $a_{p_{n_0}}$ has order a power of $3$, and we get that $a_{p_{n_0}}\notin N_{3, 2p_0, b_{n_0}}$. 
    \item If neither of the above hold, than $g$ belongs to neither of $N_{2, 2p_0, a_{p_{n_0}}}$ nor to $N_{3, 2p_0, b_{n_0}}$. 
\end{itemize}
In any of these cases,  Proposition \ref{props-index-of-N_q,n,r} guarantees that $g$ has a non-trivial image in a quotient of size at most $[G_{\calA} : N_{\max(2,3), 2p_{n_0}, b_{n_0}}] \le 48p_{n_0}^4 $. 
\end{proof}

\begin{proof}[Proof of Lemma \ref{lem-not-eff-RF}]
    Suppose that $G_{\calA}$ is effectively residually finite. Then there is a procedure which, on input a word $w$ on the generators of  $G_{\calA}$ that defines a non-identity element, produces a finite group $Q_w$ together with a morphism $\varphi_w: G_{\calA}\rightarrow Q_w$ such that $\varphi_w(w)\ne 1$. \\
    Consider now the following set $\overline{\calN}$:
    \[
    \overline{\calN}=\{n\in \nn \colon a_{p_n} \text{ has order a power of 2 in }Q_{a_{p_n}}\}
    \]
    This set is recursive, since to determine if a number $n$ belongs to $\overline{\calN}$, it suffices to apply the algorithm associated to effective residual finiteness to $a_{p_n}$, and to compute the order of $\varphi_{a_{p_n}}(a_{p_n})$ in $Q_{a_{p_n}}$.
    \\
    And by construction, we have $\calN\subseteq\overline{\calN}$ and $\calM\cap\overline{\calN}=\varnothing$, since whenever $n$ belongs to $\calM$, $a_{p_n}$ has order a power of 3 in any finite group in which it has a non-trivial image. 
    \\
    But this is a contradiction, and  $G_{\calA}$ cannot be effectively residually finite. 
\end{proof}

\begin{proof}[Proof of Lemma \ref{lem-WP_for_GA_rauzy_conj}]
    It is easy to see that the presentation \eqref{eq-calA_rauzy_conj} of $\calA$ is computable. Therefore, the decidability of the word problem for $G_{\calA}$ follows from Theorem \ref{thm-WP-in-G_A}.
\end{proof}

\section{On a problem of Nies}\label{sec-NiesPB}
\subsection{Setting}

Let $\mathfrak{S}_{\mathbb{N}}$ be the group of permutations of $\mathbb{N}$.
Let $\mathfrak{S}_{\mathbb{N}}^{+}$ be the subgroup of $\mathfrak{S}_{\mathbb{N}}$
consisting of computable bijections. 
\begin{pbm}
[Extension of a problem of Higman, posed in a ``private conversation with Belegradek and Ershov", \cite{Morozov1990}]Characterize the finitely generated subgroups of $\mathfrak{S}_{\mathbb{N}}^{+}$. 
\end{pbm}

An obvious necessary condition for a finitely generated group to embed
in $\mathfrak{S}_{\mathbb{N}}^{+}$ is to have co-semi-decidable word
problem: there should be an algorithm that stops exactly on non-identity
elements. Higman asked whether this condition was sufficient for a
finitely generated group to embed in $\mathfrak{S}_{\mathbb{N}}^{+}$.
Morozov has showed in \cite{Morozov2000} that this is not the case. 

Nies asked the same problem, but for residually finite groups. 
\begin{pbm}
[Nies, \cite{Barmpalias2025}]\label{prob:Nies} Let $G$ be a finitely generated residually
finite group with co-semi-decidable word problem. Does $G$ embed
in $\mathfrak{S}_{\mathbb{N}}^{+}$? 
\end{pbm}

The interest of this question lies in the following easy remark: if
a residually finite group with co-semi-decidable word problem is effectively
residually finite, then it does embed in $\mathfrak{S}_{\mathbb{N}}^{+}$.

Here we show: 
\begin{thm}
There exists a finitely generated residually finite group with co-semi-decidable
word problem that does not embed in $\mathfrak{S}_{\mathbb{N}}^{+}$.
\end{thm}

In particular, we obtain a new and simpler proof of the result of
Morozov. 

\subsection{Method}
We use the same obstruction as Morozov, given by the following lemma: 
\begin{lem}[\cite{Morozov2000}, Lemma 2]
Suppose that $G$ embeds in $\mathfrak{S}_{\mathbb{N}}^{+}$. Then
there is an algorithm that, on input a word $w$ which defines an
element of prime order, outputs its order. 
\end{lem}
To answer Problem \ref{prob:Nies}, we thus show the following:
\begin{props}
There exists a f.g. residually finite group $G$ with co-semi-decidable
word-problem such that no algorithm can, on input an element of prime
order, output its order.
\end{props}

\begin{proof}
We will use the group $G_{\calA}$ using $\calA\le\calB$ defined
below. 

We start by defining the group $\calB$. It is a quotient of the free
abelian group on the generators $b_{i}$, $i=1,2,3...$, obtained
simply by imposing some orders to the generators $b_{i}$. 

We do this as follows. Let $M_{1}$, $M_{2}$, ... be an enumeration
of all Turing machines. 

For each $i$, the order of $b_{i}$ depends on a run of $M_{i}$
with empty input. While this run lasts, we suppose that $b_{i}$ has
order $3$. If this element halts in $t$ steps, we choose the order
of $b_{i}$ to be the smallest prime above $\max(t,3)$. 

The word-problem is co-semi-decidable in this group. We present a
procedure that stops exactly on non-identity elements of $\calB$.

To prove that some product $\Pi b_{i}^{n_{i}}$ is non-trivial, it
is necessary and sufficient to prove that some term $b_{i}^{n_{i}}$
appearing in the product is non-trivial. To prove that some turn $b_{i}^{n_{i}}$
is non-trivial, proceed as follows.
\begin{itemize}
\item If $n_{i}$ is a multiple of $3$, run the machine $M_{i}$ until
it stops. If it never halts, the process never halts, but in this
case $b_{i}^{n_{i}}=1$, so this is a coherent behavior. If the machine
stops, we can determine exactly what the order of $b_{i}$ is, and
thus we can decide whether or not $b_{i}^{n_{i}}$ is the identity
element. 
\item If $n_{i}$ is not a multiple of $3$, run the machine $M_{i}$ for
$n_{i}$ steps. If it stops, as above, we can determine exactly what
the order of $b_{i}$, and decide whether or not $b_{i}^{n_{i}}$
is the identity element. If it does not stop in $n_{i}$ steps, the
order of $b_{i}$ is either $3$, or a prime bigger than $n_{i}$,
and thus in any case we know that $b_{i}^{n_{i}}$ is not the identity
element. 
\end{itemize}
We then define $\calA$ as in the previous proof:
\begin{align}
 \label{eq-Nies-pb}
     \calA= \langle a_i, i\in \zz \mid &[a_i, a_j]=1, i, j \in \nn; \nonumber\\
      &a_{2n}=1,\nonumber\\
     & a_n=a_{-n}, n \in \zz, \nonumber\\
    &a_{p_{n}}=b_{n}, n \in \nn,\nonumber\\
      &a_{2n+1}=\prod_{\substack{p\mid 2n+1\nonumber\\
     p \in \calP}} a_p, n\in \zz \rangle < \calB. \nonumber
 \end{align}
 The group $\calA$ is easily seen to have a basis in $\calB$, and, as in previous proofs, $\calA$ is periodic with respect to $\calB$. And thus Corollary
\ref{cor-RF_of_G_AND_upper_estimate_for_individual_RFG} shows that $G_{\calA}$ is residually finite.

$G_{\calA}$ has co-semi-decidable word-problem. Indeed, given a word
$w$, we can compute the image of $w$ in $\zz\wr\zz$, and decide
whether it is trivial there. If it is, $w$ defines an element of
$\calF'$. It is then possible to compute an expression
of $w$ as a word on the generators of $\calA$, and then as a word
on the generators of $\calB$. And because the word-problem is co-semi-decidable
in $\calB$, we can then semi-decide the statement $w\ne1$. 

Finally, no algorithm can, given an element of prime order in $G_{\calA}$,
output its order. Indeed, by construction, the element $a_{p_{i}}$
has order $3$ in $G_{\calA}$ if the $i$-th Turing machine halts,
and otherwise it has order some strictly bigger prime. 
\end{proof}


 \addtocontents{toc}{\setcounter{tocdepth}{-10}}
 
\bibliographystyle{alpha}

\bibliography{eilos}

\newpage

\noindent Arman Darbinyan: {\small\sc \\ School of Mathematical Sciences,\\ University of Southampton.\\}
{\it E-mail:} {\tt a.darbinyan@soton.ac.uk}

\vspace{3mm}

\noindent Emmanuel Rauzy:
{\small\sc \\ Institut für Theoretische Informatik, Mathematik und Operations Research,\\
Universität der Bundeswehr München,\\
\& Alexander von Humboldt Stiftung.\\}
{\it E-mail: } {\tt emmanuel.rauzy.14@normalesup.org}

\end{document}